\documentclass[11pt]{amsart}
\usepackage[margin=2.6cm]{geometry}
\usepackage{lipsum}

\usepackage{geometry}
\usepackage{latexsym}
\usepackage{amssymb}
\usepackage{amsmath}
\usepackage{color}
\usepackage{bbm}

\usepackage{graphicx}
\usepackage{tcolorbox}
\usepackage{tabularx}

\usepackage{tikz-cd}
\usepackage{tikz}
\usepackage{xparse}

\NewDocumentCommand\DownArrow{O{2.0ex} O{black}}{%
   \mathrel{\tikz[baseline] \draw [<-, line width=0.5pt, #2] (0,0) -- ++(0,#1);}
}
\NewDocumentCommand\UpArrow{O{2.0ex} O{black}}{%
   \mathrel{\tikz[baseline] \draw [->, line width=0.5pt, #2] (0,0) -- ++(0,#1);}
}

\newtheorem{theorem}{Theorem}[section]
\newtheorem{lemma}[theorem]{Lemma}
\newtheorem{proposition}[theorem]{Proposition}
\newtheorem{corollary}[theorem]{Corollary}
\newtheorem{definition}[theorem]{Definition}

\newtheorem{example}[theorem]{Example}
\newtheorem{remark}[theorem]{Remark}

\newcommand\id{\mathop{\rm id}}

\newcommand{\cl}[1]{\mathcal{#1}}
\newcommand{\bb}[1]{\mathbb{#1}}

\begin{document}
\title[The hypergraph isomorphism game, Hopf algebras and Galois extensions]{The hypergraph isomorphism game, Hopf algebras and Galois extensions}

\author[G. Baziotis]{Georgios Baziotis}
\address{Research Institute of Mathematics\\ Seoul National University, Seoul 08826, South Korea}

\address{NextQuantum Innovation Research Center, Department of Physics and Astronomy\\ Seoul National University, Seoul 08826, South Korea}
\email{baziotis@snu.ac.kr}

\author[A. Chatzinikolaou]{Alexandros Chatzinikolaou}
\address{School of Electrical and Computer Engineering\\
National Technical University of Athens\\ Athens\\ 157 80\\ Greece}

\address{Institute of Mathematics of the Polish Academy of Sciences, ul.Śniadeckich 8, 00--656 Warszawa, Poland (Current address)}
\email{achatzinikolaou@impan.pl}

\author[G. Hoefer]{Gage Hoefer}
\address{Department of Mathematics \\ Dartmouth College \\ 314 Kemeny Hall \\ Hanover \\ NH 03755 \\ USA}\email{gage.hoefer@dartmouth.edu}



\date{\today}

\begin{abstract}
   We develop an algebraic and operational framework for quantum isomorphisms of hypergraphs, 
   using tools from compact quantum group theory. We introduce a new synchronous version of the hypergraph isomorphism game whose game algebra uniformly encodes multiple notions of quantum isomorphisms of hypergraphs. We show that there exist hypergraphs that are quantum isomorphic but not classically isomorphic. For graphs, we show that the $*$-algebra of the hypergraph isomorphism game  
   is a quotient of the $*$-algebra of the graph isomorphism game. We further prove that the hypergraph game algebra forms a bi-Galois extension over the quantum automorphism groups of the underlying hypergraphs. This allows us to deduce that the algebraic notion of a quantum isomorphism of hypergraphs coincides with the operational one coming from the existence of perfect quantum strategies. Viewing games themselves as hypergraphs, we analyze isomorphisms and the transfer of strategies within this setting. Finally, we construct a $*$-algebra whose representation theory characterizes distinct classes of quantum isomorphisms between non-local games.
\end{abstract}

\maketitle

\tableofcontents

\section{Introduction}In recent years \emph{non-local games} have emerged as a central framework for studying quantum entanglement, non-locality and the limitations of classical mechanics \cite{jnppsw, lmprsstw,LMR,pal-vid}. These cooperative games are played by two non-communicating players, Alice and Bob, against a referee. In each round, the referee samples inputs $x \in X$ and $y \in Y$ from finite sets, according to a fixed probability distribution on $X \times Y$, and sends $x$ to Alice and $y$ to Bob. The players respond with outputs $a \in A$ and $b \in B$, respectively. They win the round if $\lambda(x,y,a,b) = 1$, where $\lambda : X \times Y \times A \times B \to \{0,1\}$ encodes the rules of the game.

 Although communication between the players is forbidden during play, they may agree on a joint strategy beforehand. The non-communication constraint is formalised by the class of \emph{no-signalling correlations} $(\cl{C}_{\rm ns})$, which are families of probability distributions with well-defined marginals. Within this framework, various subclasses of strategies arise from different physical models governing how the players’ systems are combined. \emph{Local strategies} ($\cl{C}_{\rm loc}$) correspond to the use of classical resources, \emph{quantum} ($\cl{C}_{\rm q}$) to finite dimensional entanglement, \emph{quantum approximate} ($\cl{C}_{\rm qa}$) to liminal entanglement and finally \emph{quantum commuting} ($\cl{C}_{\rm qc}$) arising from the commuting model of quantum mechanics. 
 
From the viewpoint of non-commutative analysis, this line of research has deepened through connections to structural problems in operator algebras, notably the equivalence between Tsirelson’s Problem in quantum information theory and Connes’ Embedding Problem for von Neumann algebras \cite{fri-CEP, jnppsw, oz, pal-vid}, culminating in the resolution of the latter in \cite{jnvwy}.

A non-local game of particular interest is the \emph{graph isomorphism game} introduced by Atserias et al. in \cite{AMRSVV}. There has been growing interest in this game both for its practical role as many problems seem to be reducible to it, and for its recently uncovered connections to quantum symmetries \cite{bigalois, LMR}. Given finite graphs $G_1,G_2$ with vertex sets $V_1,V_2$, set $V := V_1 \sqcup V_2$. Both question and answer sets for the game are $V$. In each round, the referee sends Alice a vertex $v_A \in V_i$ and Bob a vertex $v_B \in V_j$ (with $i,j \in \{1,2\}$). The players must reply with $w_A \in V_{3-i}$ and $w_B \in V_{3-j}$. They win if the relation between the vertices that lie in the same graph (equal; distinct adjacent; distinct non-adjacent) is preserved. One has that $\rm Iso(G_1,G_2)$ admits a perfect classical strategy if and only if the two graphs are isomorphic. 

The authors of \cite{AMRSVV} further defined \emph{quantum}, \emph{quantum commuting}, and \emph{no-signalling} isomorphisms via the existence of perfect strategies of the corresponding types. The graph isomorphism game is also closely related to binary constraint system (BCS) games \cite{cleve-mittal} which are of particular importance for our understanding of quantum contextuality \cite{mermin90}. The authors in \cite{AMRSVV} give a reduction from BCS games to the graph isomorphism game, yielding examples of graphs that are quantum isomorphic but not classically isomorphic, and quantum commuting isomorphic but not quantum isomorphic.

For a fixed graph $G$, playing the graph isomorphism game on $(G,G)$ shows that perfect classical strategies are in bijection with automorphisms: a winning strategy yields an element of the automorphism group $\mathrm{Aut}(G)$, and any graph automorphism trivially implements a winning strategy for the game. In the non-commutative setting, the \emph{quantum automorphism group} of a graph was introduced independently by Bichon \cite{bichon02} and Banica \cite{banica2}, providing two generalisations of the classical notion. Using Banica’s formulation, \cite{LMR,musto} established deep links between graph isomorphism games and quantum automorphism groups.
\par
Within the broader framework of quantum graphs, Brannan \emph{et al.}~\cite{bigalois} further strengthened the link between non-local games and compact quantum groups, revealing deep connections with Hopf algebras and Galois extensions.  For graphs $G_1,G_2$, the authors of~\cite{bigalois} showed that the $*$-algebra
\(  \mathcal A\!\left(\mathrm{Iso}(G_1,G_2)\right)\)~\cite{hmps},
serve as the \emph{quantum isomorphism space} of $G_1$ and $G_2$ (see also~\cite{bks}), that is, a non-commutative analogue of the algebra of functions on the set of graph isomorphisms. They proved that the non-triviality of this game $*$-algebra guarantees the existence of perfect quantum commuting strategies for the graph isomorphism game. Their results established a powerful bridge between the $*$-algebraic and ${\rm C}^*$-algebraic settings, mediated by techniques from Hopf algebras and bi-Galois extensions, tools that will play a central role in the present work.

\par
As in the case of graphs, hypergraphs also exhibit symmetries beyond the classical setting. The \emph{quantum automorphism group of a hypergraph}, introduced by Faross \cite{frs}, generalises  Bichon’s construction in the graph case and can be viewed as a non-commutative analogue of the algebra of continuous functions on the automorphism group of a hypergraph. From an operational perspective, quantum isomorphisms of hypergraphs were investigated through the \emph{hypergraph isomorphism game}, introduced by the third author and Todorov in \cite{ht_one}. There, a hierarchy of isomorphism types between hypergraphs is established via winning bicorrelations \cite{bhtt} for this game. In the hypergraph isomorphism game of \cite{ht_one}, Alice receives vertices from a fixed vertex set and responds with vertices from the opposite set, while Bob receives edges from a fixed edge set and responds with edges from the opposite set. They win if and only if their questions and answers preserve the incidence relations of the hypergraphs. 

Besides their structural richness, the study of hypergraph isomorphisms is motivated by their connection to isomorphisms between non-local games and transfers of strategies. In particular, \cite{ht_one} shows that every non-local game can be encoded as a hypergraph via the support of the rule function. This correspondence enables the definition of various isomorphism types between non-local games and allows one to analyse the transfer of strategies between them using the simulation paradigm in quantum information theory. In this sense, the hypergraph isomorphism game between the encoded non-local games may be viewed as a ``supergame'' played by the referees of the respective games, who win precisely when both underlying games yield the same outcome (win/lose).

\par
In the present work, we further develop the connections between non-commutative analysis, non-local games, and compact quantum groups. We introduce a new \emph{synchronous} version—that is,  the players must produce identical answers when given identical questions—of the \emph{hypergraph isomorphism game}, extending the framework of \cite{ht_one} and revealing strong links with the quantum automorphism group of a hypergraph \cite{frs} by building on the approach of \cite{bigalois}. 

Given two hypergraphs $\Lambda_i = (V_i, E_i)$, $i = 1,2$, set $V := V_1 \sqcup V_2$ and $E := E_1 \sqcup E_2$. The hypergraph isomorphism game ${\rm HypIso}(\Lambda_1,\Lambda_2)$ is played with question and answer sets $V \sqcup E$: in each round, the referee sends Alice and Bob either a vertex or an edge. Upon receiving their question, each player must respond with an element of the same type (vertex or edge) from the opposite hypergraph. They win if and only if all question–answer pairs preserve the corresponding relations. More precisely, 
\begin{itemize}
    \item[(i)] if both questions and answers are vertices, then vertices within the same hypergraph must satisfy the same relations (equal; distinct and adjacent; distinct and non-adjacent) with respect to the opposite hypergraph, where adjacency means belonging to a common hyperedge;
    \item[(ii)] if both are edges, then the edges within the same hypergraph must preserve the relations (equal; intersecting; disjoint) with edges from the opposite hypergraph;
    \item[(iii)] if one question (and answer) is a vertex and the other an edge, the respective incidence relations must be maintained.
\end{itemize}

Two hypergraphs $\Lambda_1$ and $\Lambda_2$ are said to be \emph{isomorphic}, written $\Lambda_1 \cong \Lambda_2$, if there exist bijections between their vertex and edge sets that preserve the incidence relations. Equivalently, there exist permutation matrices $P_V \in M_{V_1,V_2}(\mathbb C)$ and $P_E \in M_{E_1,E_2}(\mathbb C)$ such that
\( 
A_{\Lambda_1} P_E = P_V A_{\Lambda_2},
\)
where $A_{\Lambda_i}$ denotes the incidence matrix of $\Lambda_i$. Following Faross~\cite{frs}, a natural quantisation of this notion replaces the permutation matrices by \emph{quantum permutation matrices} (or \emph{magic unitaries}) with entries in a non-commutative algebra satisfying the same intertwining relation with the incidence matrices.

To every \emph{synchronous} non-local game $\mathcal G$ one can associate a unital $*$-algebra $\mathcal A(\mathcal G)$ encoding the perfect strategies for $\mathcal G$~\cite{hmps}.
One of the main advantages of the new formulation of the hypergraph isomorphism game, is that it is inherently synchronous, and its associated game $*$-algebra encodes the various types of isomorphisms between the hypergraphs through its representations. In particular, we provide an explicit presentation of this game $*$-algebra as the $*$-algebra generated by the entries of two magic unitaries, indexed by the vertices and edges respectively, which satisfy an intertwining relation with the incidence matrices of the hypergraphs.

Intuitively, the hypergraph isomorphism game captures this notion of (quantum) isomorphism by requiring the players to respond with elements of the same type but from the opposite hypergraph leading naturally to the appearance of two matrices in the game algebra—one indexed by vertices and the other by edges. The vertex relations in the game enforce that the vertex matrix forms a magic unitary, while the edge relations impose the same property on the edge matrix. Finally, the requirement that incidence relations be preserved translates exactly into the intertwining relation between the two.

\par

This first step yields a direct connection to the theory of quantum symmetries of hypergraphs. Indeed, for a single hypergraph $\Lambda := \Lambda_1 = \Lambda_2$, the game $*$-algebra $\mathcal A(\mathrm{HypIso}(\Lambda,\Lambda))$ coincides precisely with the Hopf $*$-algebra associated to the quantum automorphism group $\mathrm{Aut}^{+}(\Lambda)$ \cite{frs}. For a simple graph $G$ the game algebra reduces to Bichon’s quantum automorphism Hopf $*$-algebra \cite{bichon02}:
\( 
\mathcal A({\rm HypIso}(G,G)) \;\cong\; \mathcal O({\rm Aut^+_{Bic}}(G)),
\)
thus recovering the result of Faross \cite[Theorem~4.3]{frs} in the $*$-algebraic setting. The distinction between the hypergraph and graph isomorphism games is reflected, in the graph case, by the difference between Banica-type and Bichon-type quantum symmetry formalisms.

 This new game generalises the one in \cite{ht_one} and their perfect strategies naturally describe quantum isomorphisms between hypergraphs. For $t\in\{\rm loc, q, qa, qc\}$, we say that two hypergraphs $\Lambda_1$ and $\Lambda_2$, are $t$-isomorphic ($\Lambda_1\cong_{\rm t}\Lambda_2$) if the $\Lambda_1$-$\Lambda_2$ isomorphism game has a perfect $t$-strategy. Moreover, since the hypergraph isomorphism game is synchronous, one can define additional notions of isomorphisms between $\Lambda_1$ and $\Lambda_2$ relying on the \emph{game algebra} of $\operatorname{HypIso}$. In a similar fashion with \cite{bigalois} in the case of (quantum) graphs, we say that $\operatorname{HypIso}(\Lambda_1,\Lambda_2)$ has a \emph{perfect A*-strategy} ($\Lambda_1\cong_{A^*}\Lambda_2$) if the game algebra is non-trivial, and that it has a \emph{perfect {\rm C}*-strategy} ($\Lambda_1\cong_{\rm C^*}\Lambda_2$) if the game algebra admits  a unital *-representation into $\cl{B}(H)$ for some Hilbert space $H$. 
As a result, the following chain of implications is immediate, except for the one marked \((\dagger)\):
\begin{figure}[ht]
\begin{tikzcd}[column sep=huge, row sep=large] 
\Lambda_{1}\cong_{\mathrm{loc}}\Lambda_{2}
  \arrow[r, phantom, "\implies"] &
\Lambda_{1}\cong_{\mathrm{q}}\Lambda_{2}
  \arrow[r, phantom, "\implies"] &
\Lambda_{1}\cong_{\mathrm{qa}}\Lambda_{2}
  \arrow[r, phantom, "\implies"] &
\Lambda_{1}\cong_{\mathrm{qc}}\Lambda_{2} \\
& &
\Lambda_{1}\cong_{\mathrm{C^*}}\Lambda_{2}
  \arrow[r, phantom, "\implies"] &
\Lambda_{1}\cong_{\mathrm{A^*}}\Lambda_{2}
\arrow[from=1-4, to=2-3, phantom, "\impliedby" sloped]
\arrow[from=2-4, to=1-4, phantom,
  "{\rotatebox{90}{$\implies$}}"{pos=.5, swap},
  "{(\dagger)}"{pos=.5, xshift=2.0ex}]
\end{tikzcd}
\caption{Implication relations between isomorphisms of $ \Lambda_1, \Lambda_2$.}
\label{fig:iso-implications}
\end{figure}
\noindent The step \((\dagger)\), i.e.\ \(\mathrm{A}^* \Rightarrow \mathrm{qc}\), was established in \cite{bigalois} for the graph isomorphism game. In the present work we prove that the implication \((\dagger)\) holds for the hypergraph isomorphism game as well. Following \cite{bigalois}, our strategy is to identify the game $*$-algebra as an \(\mathcal O(\Lambda_1)\)–\(\mathcal O(\Lambda_2)\) bi-Galois extension (with respect to the Hopf $*$-algebras of quantum automorphisms) and then invoke general results on bi-Galois extensions, relating the existence of ${\rm C}^*$-representations to the presence of invariant states.

\par
The paper is organised as follows. Section~\ref{s_preliminaries} recalls the requisite background on non-local games and their strategy classes, compact quantum groups, the associated Hopf $*$-algebras, and related notions. In Section~\ref{s_hig} we present the hypergraph isomorphism game, characterise the associated game algebra, and describe the different notions of hypergraph isomorphisms arising from perfect strategies.

We then relate the hypergraph and graph isomorphism games. For simple graphs $G_1,G_2$, we show that the hypergraph game $*$-algebra $\mathcal A({\rm HypIso}(G_1,G_2))$ is a quotient of the graph game $*$-algebra $\mathcal A({\rm Iso}(G_1,G_2))$. As a consequence, if two graphs are $\rm t$-isomorphic as hypergraphs, they are also $\rm t$-isomorphic as graphs. We also show that the two games are in general distinct by exhibiting a perfect quantum strategy for the graph isomorphism game that is not a perfect strategy for the hypergraph game. Leveraging these connections, we present hypergraphs $\Lambda_1,\Lambda_2$ that are quantum isomorphic but not (classically) isomorphic.

\par In Section \ref{sec_bigalois} we focus on showing  the implication \((\dagger)\) in figure \ref{fig:iso-implications}. More specifically that $ \Lambda_1 \cong_{A^*} \Lambda_2$ implies $\Lambda_1 \cong_{\rm qc} \Lambda_2 $. We denote the game algebra $\cl{A}(\operatorname{HypIso}(\Lambda_1,\Lambda_2))$ of the hypergraph isomorphism game by $\cl{O}(\Lambda_1,\Lambda_2)$ and think of it as the quantum isomorphism space of $\Lambda_1$ and $\Lambda_2$. Similarly with the (quantum) graph treatment in \cite{bigalois}, we show that the quantum isomorphism space $\cl{O}(\Lambda_1,\Lambda_2)$ defines a left (resp. right) *-comodule algebra structure over $\cl{O}(\Lambda_1)$ (resp. $\cl{O}(\Lambda_2)$) and is in fact a $\cl{O}(\Lambda_1)-\cl{O}(\Lambda_2)$ bi-Galois extension. Furthermore, we show the existence of bi-invariant states $\tau:\cl{O}(\Lambda_1,\Lambda_2)\rightarrow\bb{C}$ and prove that the non-triviality of $\cl{O}(\Lambda_1,\Lambda_2)$ is equivalent with the existence of quantum isomorphisms between $\Lambda_1$ and $\Lambda_2$.

\par
As an application, in Section~\ref{s_nlg_trans} we revisit the notion of non-local game isomorphisms from \cite{ht_one} and give an algebraic characterisation of ${\rm t}$-isomorphic non-local games. Specifically, we construct a universal unital \emph{no-signalling} $*$-algebra which, when non-zero, encodes quantum isomorphisms between the games. Viewing non-local games $\mathcal G_1$ and $\mathcal G_2$ as hypergraphs, we show that the no-signalling algebra $\mathcal O_{\rm NS}(\mathcal G_1,\mathcal G_2)$ is a bi-Galois extension over quotient Hopf $*$-algebras associated to the quantum automorphism groups of the two games, whose generators satisfy strengthened no-signalling relations. Representations of this algebra of various types are in bijection with strategies of the corresponding types implementing game isomorphisms.

In particular, this framework enables transfer of strategies: if $\mathcal G_1$ and $\mathcal G_2$ are ${\rm t}$-isomorphic, then $\mathcal G_1$ admits a perfect ${\rm t}$-strategy if and only if $\mathcal G_2$ admits a perfect ${\rm t}$-strategy.

\medskip

\noindent 
{\bf Acknowledgments.} The authors would like to thank the anonymous referee for their careful reading and suggestions. Additionally, the authors would like to thank Ivan G. Todorov and Lyudmila Turowska for helpful comments and suggestions and Junichiro Matsuda for useful discussions. G.B. was supported by the National Research Foundation of Korea (NRF) grant funded by the Korea government (MSIT) (No.RS-2025-00561391, No.RS-2024-00413957 and RS-2022-NR069971) and by NSF grants 
2115071 and 2154459.
A.C. was supported by the European Union -- Next Generation EU (Implementation Body: HFRI. Project
name: Noncommutative Analysis: Operator Systems and Nonlocality. HFRI
Project Number: 015825). G.H. was supported by the U.S.\ Department of Defense, Basic Research Office, under Vannevar Bush Faculty Fellowship grant N00014-21-1-2946, managed by the Office of Naval Research. The authors would like to acknowledge support by the Swedish Research Council under grant no. 2021-06594 while they were in residence at Institut Mittag-Leffler in Djursholm, Sweden during the Spring semester of 2026.

\section{Notation and 
preliminaries}\label{s_preliminaries}

For a given Hilbert space $H$, we denote by $\cl{B}(H)$ the C$^*$-algebra of bounded linear operators on $H$ and write $I=I_H$ for the identity operator. We further assume that all inner products are linear in the first coordinate. For a finite set $X$, we use $M_{X}$ to denote the full matrix algebra of $|X|\times |X|$ matrices with complex entries, and $\cl{D}_{X}$ to denote the subalgebra of diagonal matrices. More generally, for a ${*}$-algebra $\cl{A}$, $M_{X}(\cl{A})$ denotes the $|X|\times |X|$ matrices with entries in $\cl{A}$. For a unital $*$-algebra $\cl{A}$, we say a matrix $u = (u_{i,j}) \in M_{n, m}(\cl{A})$ is unitary if $u^*u =I_{M_{m}(\cl A)}$ and $ uu^* =I_{M_{n}(\cl A)}$. For $x \in X$, let $\delta_{x}$ denote the $x^{\rm th}$-basis element in $\cl{D}_{X}$.
A \textit{positive operator-valued measure (POVM)} is a  family $(A_{i})_{i=1}^{k}$ of positive operators acting on $H$, such that $\sum_{i=1}^{k}A_{i} = I$. If, in addition, $A_{i}$ is a projection for any $i\in[k]$, we call it a \textit{projection-valued measure (PVM)}. 

\smallskip

\textbf{Non-local games and correlation sets.} Let $X, Y, A, B$ be finite sets. A collection of (conditional) probabilities $p = \{p(a, b|x, y): \; x \in X, y \in Y, a \in A, b \in B\}$ is called a \textit{no-signalling (NS) correlation} if
\begin{gather*}
    \sum\limits_{b \in B}p(a, b|x, y) = \sum\limits_{b\in B}p(a, b|x, y'), \;\;\;\; {\rm and} \;\;\;\;
    \sum\limits_{a \in A}p(a, b|x, y) = \sum\limits_{a \in A}p(a, b|x', y),
\end{gather*}
\noindent for any choice of $x, x' \in X, y, y' \in Y, a \in A, b \in B$.

A no-signalling correlation $p$ is called \textit{quantum commuting} if there exists a Hilbert space $H$, a unit vector $\xi \in H$, and POVM's $(E_{x, a})_{a \in A}$ and $(F_{y, b})_{b \in B}$ acting on $H$, where $E_{x, a}F_{y, b} = F_{y, b}E_{x, a}$ for all $x \in X, y \in Y, a \in A, b \in B$ such that
\begin{gather*}
    p(a, b|x, y) = \langle E_{x, a}F_{y, b}\xi, \xi\rangle, \;\;\;\; x \in X, y \in Y, a \in A, b \in B.
\end{gather*}
We call $p$ \textit{quantum}, if it is quantum commuting and $H$ can be chosen of the form $H = H_{A}\otimes H_{B}$ in such a way that $H_{A}, H_{B}$ are finite-dimensional, with $E_{x, a} = E_{x, a}'\otimes I_{B}$ and $F_{y, b} = I_{A}\otimes F_{y, b}'$ for $x \in X, y \in Y, a \in A, b \in B$. The no-signalling correlation $p$ is \textit{approximately quantum} if it is the limit of quantum correlations, and \textit{local} if $p = \sum_{i=1}^{k}\lambda_{i}p_{i}^{1}\otimes p_{i}^{2}$ as a convex combination, where $p_{i}^{1}(\cdot|x)$ (resp. $p_{i}^{2}(\cdot|y)$) is a probability distribution over $A$ (resp. over $B$) for each choice of $x \in X$ (resp. $y \in Y$), $i = 1, \hdots, k$. We denote the subclasses of local, quantum, approximately quantum, quantum commuting, and no-signalling correlations by $\cl{C}_{\rm loc}, \cl{C}_{\rm q}, \cl{C}_{\rm qa}, \cl{C}_{\rm qc}$, and $\cl{C}_{\rm ns}$, respectively (where we drop the explicit dependence on $X, Y, A$, and $B$ whenever understood). 

A \textit{non-local game over $(X, Y, A, B)$} is a tuple $\cl G = (X, Y, A, B, \lambda)$ where $X, Y, A, B$ are finite sets, and $\lambda: X\times Y\times A\times B\rightarrow \{0, 1\}$. Let $\cl G = (X, Y, A, B, \lambda)$ be a non-local game; for ${\rm t} \in \{\rm loc, q, qa, qc, ns\}$, we say that a correlation $p \in \mathcal{C}_{\rm t}$ is a \textit{perfect ${\rm t}$-strategy} for $\cl G$ if $p(a, b|x, y) = 0$ whenever $\lambda(x, y, a, b) = 0$, for $x \in X, y \in Y, a \in A, b \in B$. We denote the collection of all perfect ${\rm t}$-strategies for $\cl G$ by $\cl{C}_{\rm t}(\cl G)$. 

A non-local game is \textit{synchronous} if $X = Y, A = B$ and $\lambda(x, x, a, b) = \delta_{a,b}$ for $a, b \in A$ and all $x \in X$ and \textit{bisynchronous} if  it is synchronous and $ \lambda(x,y,a,a)= \delta_{x,y}$ for $ x,y\in X$ and all $ a\in A$ (see \cite{paulsen_rahaman}). Intuitively, those are non-local games for which our players must respond with the same answer, given they were sent the same question. Note that if $p \in \mathcal{C}_{\rm t}$ is a perfect strategy for a synchronous game, then $p(a, b|x, x) = 0$ whenever $a \neq b$, for all $x \in X$ and $a, b \in A$. 
We briefly discuss this characterization in the remainder of this subsection.

Let $\cl G = (X, A, \lambda)$ be a synchronous game. The \textit{game algebra associated to $\cl G$} (see \cite{hmps, kps, op}) is the  $*$-algebra $\mathcal{A}(\cl G)$ generated by elements $\{e_{x, a}: \; x \in X, \; a \in A\}$ subject to the following conditions:
\begin{itemize}
    \item[(i)] $e_{x, a}^{2} = e_{x, a}^{*} = e_{x, a}$ for all $x \in X, a \in A$;
    \item[(ii)] $\sum_{a \in A}e_{x, a} = 1$, for each $x \in X$;
    \item[(iii)] $e_{x, a}e_{y, b} = 0$ whenever $\lambda(x, y, a, b) = 0$, for $x, y \in X, a, b \in A$. 
\end{itemize}
Recall the following;
\begin{theorem}[\cite{hmps, kps}] \label{th:synchro_charact}
Let $\cl G = (X, A, \lambda)$ be a synchronous game. Then 
\begin{itemize}
    \item[(i)] there exists a perfect ${\rm qc}$-strategy for $\cl{G}$ if and only if there exists a unital $*$-homomorphism $\cl{A}(\cl{G})\rightarrow \cl{A}$, where $\cl{A}$ is a unital ${\rm C}^{*}$-algebra endowed with a faithful tracial state;
    \item[(ii)] there exists a perfect ${\rm qa}$-strategy for $\cl{G}$ if and only if there exists a unital $*$-homomorphism $\cl{A}(\cl{G})\rightarrow \cl{R}^{\cl{U}}$, where $\cl{U}$ is a free ultrafilter along $\bb{N}$ and $\cl{R}$ is the hyperfinite ${\rm II}_{1}$ factor;
    \item[(iii)] there exists a perfect ${\rm q}$-strategy for $\cl{G}$ if and only if there exists a unital $*$-homomorphism $\cl{A}(\cl{G})\rightarrow \cl{B}(H)$ for some non-zero, finite-dimensional Hilbert space $H$;
    \item[(iv)] there exists a perfect ${\rm loc}$-strategy for $\cl{G}$ if and only if there exists a unital $*$-homomorphism $\cl{A}(\cl{G})\rightarrow \bb{C}$.
\end{itemize}
\end{theorem}

\noindent Note that it is not guaranteed that $\mathcal{A}(\cl G) \neq \{0\}$ for any synchronous game $\cl G$. As previously indicated, when $\cl{A}(\cl G)$ is \textit{non-trivial} this $*$-algebra encodes the perfect winning strategies for a synchronous game $\cl G$; see \cite[Theorem 3.2]{hmps} and \cite[Theorem 3.6]{kps}. The representation theory of $\cl{A}(\cl{G})$ is delicate, and depends on the choice of synchronous game $\cl{G}$: there exist synchronous games $\cl G$ for which $\cl{A}(\cl G) \neq \{0\}$, and may be represented concretely $\cl{A}(\cl G)\rightarrow \cl{B}(H)$ for some Hilbert space $H$, but there exist no tracial states on $\cl{A}(\cl G)$ (implying no winning ${\rm qc}$-strategies--- see \cite{ps}). Similarly, there exist synchronous games $\cl G$ for which $\cl{A}(\cl G) \neq \{0\}$, but cannot be concretely represented on $\cl{B}(H)$ for any choice of $H$; this indicates a severe lack of perfect strategies of various types for $\cl G$ (see \cite{hmps}).

\smallskip

\textbf{Compact quantum groups and Hopf algebras.} In what follows, $\otimes$ denotes the minimal tensor product of ${\rm C}^{*}$-algebras unless otherwise specified. We refer the reader to \cite{timmermann} for a more detailed exposition.
\smallskip

A \textit{compact quantum group (CQG)} $\bb{G}$ is a pair $(C(\bb{G}), \Delta)$, where $C(\bb{G})$ is a unital ${\rm C}^{*}$-algebra and $\Delta: C(\bb{G})\rightarrow C(\bb{G})\otimes C(\bb{G})$ is a unital $*$-homomorphism such that:
\begin{itemize}
  \item[(i)] $(\mathrm{id}\otimes \Delta)\Delta = (\Delta\otimes \mathrm{id})\Delta$;
  \item[(ii)] we have the following density conditions:
  \[
    \overline{\mathrm{span}}\bigl\{ \Delta(C(\bb{G}))(1\otimes C(\bb{G})) \bigr\}
    \;=\; C(\bb{G})\otimes C(\bb{G})
    \;=\;
    \overline{\mathrm{span}}\bigl\{ (C(\bb{G})\otimes 1)\Delta(C(\bb{G})) \bigr\}.
  \]
\end{itemize}

Every compact quantum group admits a unique state $\varphi_{\bb{G}}: C(\bb{G})\rightarrow \bb{C}$, called the \textit{Haar state}, such that
\[
(\mathrm{id}\otimes \varphi_{\bb{G}})\Delta(a) = \varphi_{\bb{G}}(a)\,1
\;=\;
(\varphi_{\bb{G}}\otimes \mathrm{id})\Delta(a)
\qquad (a \in C(\bb{G})).
\]
If the Haar state is tracial, then $\bb{G}$ is said to be of \textit{Kac type}.

Given two compact quantum groups $\bb{G} = (C(\bb{G}), \Delta_{\bb{G}})$ and $\bb{H} = (C(\bb{H}), \Delta_{\bb{H}})$, a unital $*$-homomorphism $\varphi: C(\bb{G})\rightarrow C(\bb{H})$ is a \emph{morphism of compact quantum groups} if
\[
(\varphi\otimes \varphi)\circ \Delta_{\bb{G}} \;=\; \Delta_{\bb{H}}\circ \varphi.
\]
If $\varphi$ is surjective, we write $\bb H \subseteq \bb G$ and call $\bb H$ a subgroup of $\bb G$; if $\varphi$ is injective, we call $\bb H$ a quotient of $\bb G$; and if $\varphi$ is bijective, then $\bb{G}$ and $\bb{H}$ are \emph{isomorphic as compact quantum groups}.

\smallskip

A \emph{compact matrix quantum group} is a triple $(\cl{A}, \Delta, u)$ where $\cl{A}$ is a unital ${\rm C}^{*}$-algebra, $\Delta: \cl{A}\rightarrow \cl{A}\otimes \cl{A}$ is a unital $*$-homomorphism, and $u=(u_{ij})\in M_{n}(\cl{A})$ is unitary such that:
\begin{itemize}
  \item[(i)] $\Delta(u_{ij}) = \sum_{k=1}^{n}u_{ik}\otimes u_{kj}$ for all $1\le i,j\le n$;
  \item[(ii)] $\bar u := (u_{ij}^{*})_{i,j}$ is invertible in $M_{n}(\cl{A})$;
  \item[(iii)] the entries $\{u_{ij}\}_{1\le i,j\le n}$ generate $\cl{A}$ as a ${\rm C}^{*}$-algebra.
\end{itemize}
The unitary matrix $u$ is called the \emph{fundamental representation}.

\smallskip

Given a compact matrix quantum group $(\cl{A}, \Delta, u)$, the $*$-algebra generated by the matrix coefficients $\{u_{ij}\}$ is denoted by $\mathcal{O}(\bb{G})$. 

For compact (matrix) quantum groups $\bb{G}=(C(\bb{G}),\Delta_{\bb G},u)$ and $\bb{H}=(C(\bb{H}),\Delta_{\bb H},v)$, the pair
\[
\bb{G}\ast \bb{H} \;:=\; \bigl(C(\bb{G})\ast C(\bb{H}),\, \Delta_{\bb G}\ast \Delta_{\bb H},\, u\oplus v\bigr),
\]
where the free product is taken in the category of unital ${\rm C}^{*}$-algebras and $u,v$ are viewed inside $\,C(\bb G)\ast C(\bb H)\,$ via the canonical inclusions, is again a compact matrix quantum group \cite{wang}.

A {\it magic unitary} or {\it quantum permutation matrix} is a unitary matrix $u=(u_{i,j})\in M_{n}(\cl A)$ such that its entries are selfadjoint idempotents and the rows and columns are orthogonal; that is, $u_{i,j}= u_{i,j}^* = u_{i,j}^2$, $\sum_{i=1}^{n}u_{i,j}=\sum_{j=1}^{n}u_{i,j}=1$ for all $ i,j$, and $ u_{i,j}u_{i,j'}=0$ if $ j\neq j'$ and $u_{i,j}u_{i',j}=0$ if $ i \neq i'$. In the case when $\cl{A}$ is a C$^{*}$-algebra, if $u \in M_{n, m}(\cl{A})$ is a magic unitary then orthogonality of the rows and columns is automatic, and does not need to be assumed. An example of particular importance here is the \emph{quantum symmetric group} $S_{X}^{+}$ of a finite set $X$: it is the compact quantum group whose fundamental representation $u=(u_{xx'})\in M_{X}( C(S_{X}^{+}))$ is a magic unitary (quantum permutation). 

\smallskip

A  \emph{Hopf $*$-algebra} is a quadruple $(\cl A,\Delta,\varepsilon,S)$ consisting of a unital $*$-algebra $\cl A$, a unital $*$-homomorphism $\Delta:\cl A\to \cl A\otimes \cl A$  such that $(\mathrm{id}\otimes \Delta)\Delta=(\Delta\otimes \mathrm{id})\Delta$ and  morphisms  $\varepsilon:\cl A\to\bb C$ (the \emph{counit}) and $S:\cl A\to \cl A$ (the \emph{antipode}) such that, for all $a\in \cl A$,
\begin{itemize}
  \item[(i)] $(\varepsilon\otimes \mathrm{id})\Delta(a)=(\mathrm{id}\otimes \varepsilon)\Delta(a)=a$;
  \item[(ii)] $m(S\otimes \mathrm{id})\Delta(a)=m(\mathrm{id}\otimes S)\Delta(a)=\varepsilon(a)\,1$,
\end{itemize}
where $m:\cl A\otimes \cl A\to \cl A$ is the multiplication map.

\smallskip

\begin{remark}\label{r_hopf_alg}\rm
By \cite{timmermann}, if $\bb{G}$ is a compact matrix quantum group, then there exist a unital $*$-homomorphism $\varepsilon:\mathcal{O}(\bb G)\to\bb C$ and a unital anti-homomorphism $S:\mathcal{O}(\bb G)\to\mathcal{O}(\bb G)$ such that
\[
\varepsilon(u_{ij})=\delta_{ij},\qquad S(u_{ij})=u_{ji}^{*}\qquad (1\le i,j\le n),
\]
making $(\mathcal{O}(\bb G),\Delta,\varepsilon,S)$ a Hopf $*$-algebra.
\end{remark}

Let $F\in {\rm GL}_n(\bb C)$ be invertible. The \emph{universal unitary quantum group} $A^{0}_{u}(F)$ \cite{vdw} is the unital $*$-algebra generated by the entries $u_{ij}$ of an $n\times n$ matrix $u=(u_{ij})$ subject to the relations that both $u$ and
\[
F\,\bar u\,F^{-1},\qquad \text{with }\ \bar u:=(u_{ij}^{*}),
\]
are unitary. It becomes a compact quantum group with coproduct $\Delta(u_{ij})=\sum_{k=1}^{n}u_{ik}\otimes u_{kj}$, counit $\varepsilon(u_{ij})=\delta_{ij}$, and antipode $S(u_{ij})=u_{ji}^{*}$. For $F=I_n$ we obtain the \emph{universal unitary quantum group} $A^0_u(n)$. Given that $A^0_u$ is a *-algebra, we shall use $A_u$ to denote its C*-completion.

Let $G$ be a group acting on a set $X$; the action is \textit{free} if $g\cdot x = x$ implies $g = e_{G}$, and \textit{transitive} if for any $x, y \in X$ there exists a $g \in G$ such that $g \cdot x = y$. The space $X$ is a \emph{torsor} (principal homogeneous space) for $G$ if the action is free and transitive. Galois extensions (introduced in \cite{cs}) are its quantum analogues in the setting of compact quantum groups/Hopf $*$-algebras; see also \cite{bichon99, ulbrich}. Bi-Galois extensions (comodules with commuting left/right coactions) will play a role below.

\smallskip

Let $\cl A$ be a Hopf $*$-algebra. A unital $*$-algebra $\cl Z$ is:
\begin{itemize}
  \item[(i)] a \emph{left $\cl A$-$*$-comodule algebra} if there exists a unital $*$-homomorphism $\delta:\cl Z\to \cl A\otimes \cl Z$ with
  \[
  (\mathrm{id}\otimes \delta)\circ \delta=(\Delta\otimes \mathrm{id})\circ \delta,
  \qquad
  (\varepsilon\otimes \mathrm{id})\circ \delta=\mathrm{id}_{\cl Z};
  \]
  \item[(ii)] a \emph{right $\cl A$-$*$-comodule algebra} if there exists a unital $*$-homomorphism $\gamma:\cl Z\to \cl Z\otimes \cl A$ with
  \[
  (\gamma\otimes \mathrm{id})\circ \gamma=(\mathrm{id}\otimes \Delta)\circ \gamma,
  \qquad
  (\mathrm{id}\otimes \varepsilon)\circ \gamma=\mathrm{id}_{\cl Z};
  \]
\end{itemize}

\smallskip

Let $\cl A$ and $\cl B$ be Hopf $*$-algebras. A unital $*$-algebra $\cl Z$ is an \emph{$\cl A$-$\cl B$ bi-Galois extension} if:
\begin{enumerate}
  \item[(i)] $\cl Z$ is a left $\cl A$-$*$-comodule algebra via $\delta:\cl Z\to \cl A\otimes \cl Z$;
  \item[(ii)] $\cl Z$ is a right $\cl B$-$*$-comodule algebra via $\gamma:\cl Z\to \cl Z\otimes \cl B$;
  \item[(iii)] the coactions commute: $(\delta\otimes \mathrm{id})\circ \gamma=(\mathrm{id}\otimes \gamma)\circ \delta$;
  \item[(iv)] the canonical maps
  \begin{align*}
  \kappa_\delta:\cl Z\otimes \cl Z\to \cl A\otimes \cl Z,\quad z\otimes z'\mapsto \delta(z)(1\otimes z')\\
  \kappa_\gamma:\cl Z\otimes \cl Z\to \cl Z\otimes \cl B,\quad z\otimes z'\mapsto (z\otimes 1)\gamma(z'),
  \end{align*}
  are bijections.
\end{enumerate}

In the sequel we will be interested in bi-Galois extensions that can be represented in some $ {\rm C}^*$-algebra equipped with a trace. The existence of such representations will involve a criterion that guarantees the existence of invariant states on bi-Galois extensions.
For a linear functional $\tau:\cl Z\to \bb C$ we say:
\begin{itemize}
  \item[(i)] $\tau$ is a \emph{state} if $\tau(z^{*}z)\ge 0$ for all $z$ and $\tau(1)=1$;
  \item[(ii)] $\tau$ is a \emph{trace} if it is a state and $\tau(z_1 z_2)=\tau(z_2 z_1)$ for all $z_1,z_2$;
  \item[(iii)] $\tau$ is \emph{left invariant} if $(\mathrm{id}\otimes \tau)\circ \delta(z)=\tau(z)\,1_{\cl A}$, for all $z$;
  \item[(iv)] $\tau$ is \emph{right invariant} if $(\tau\otimes \mathrm{id})\circ \gamma=\tau(z)\,1_{\cl{B}}$, for all $z$;
  \item[(v)] $\tau$ is \emph{bi-invariant} if it is both left and right invariant;
  \item[(vi)] $\tau$ is \emph{faithful} if $\tau(z^{*}z)=0\Rightarrow z=0$.
\end{itemize}

\smallskip

Let $F\in {\rm GL}_n(\bb C)$ and $G\in {\rm GL}_m(\bb C)$ be invertible. The \emph{unital $*$-algebra} $A_u^0(F,G)$ \cite{bichon99} is the unital $*$-algebra generated by the entries $u_{ij}$ of an $n\times m$ matrix $u=(u_{ij})$ such that both $u$ and $F\,\bar u\,G^{-1}$ are unitary. For $ F= I_n$ and $ G= I_m$ we obtain the unital $*$-algebra $A_u^0(n,m)$. If, in addition, $n = m$ with $F = I_{n} = G$ we denote $A_{u}^{0}(n) := A_{u}^{0}(n, n)$. 

\smallskip

\textbf{Quantum automorphisms of hypergraphs.} A \textit{hypergraph} is a pair $ \Lambda=(V,E)$, where $ V$ is a finite set and 
$E$ is a non-empty set of subsets of $V$.  The elements of $V$ are called vertices, and the elements of $E$ hyperedges (or simply edges). A hypergraph is called \textit{full} if  $ \cup_{e \in E}e = V$. The hypergraphs considered in this work will all be assumed full. Moreover, we note that the hypergraphs are not assumed to be multiplicity-free and may contain multiple duplicate hyperedges.

Let $\Lambda=(V,E)$ be a hypergraph. Its incidence matrix $A_{\Lambda}$ is the $V\times E$ matrix
\[
(A_{\Lambda})_{v,e} :=
\begin{cases}
1, & v\in e,\\
0, & \text{otherwise.}
\end{cases}
\]

The \emph{quantum automorphism group of a hypergraph} $\Lambda$, denoted $\mathrm{Aut}^{+}(\Lambda)$, was introduced by Faross in \cite{frs}; it is the compact matrix quantum group $(C(\mathrm{Aut}^{+}(\Lambda)),\Delta,\,u_V\oplus u_E)$ whose ${\rm C}^*$-algebra is the universal unital ${\rm C}^*$-algebra generated by
\begin{itemize}
  \item[(i)] a quantum permutation matrix $u_V=(u_{v,v'})_{v,v'\in V}$ on vertices,
  \item[(ii)] a quantum permutation matrix $u_E=(u_{e,e'})_{e,e'\in E}$ on edges,
\end{itemize}
subject to the intertwining relation $A_\Lambda\,u_E=u_V\,A_\Lambda$.

For the compact matrix quantum group $\mathrm{Aut}^{+}(\Lambda)$, we write $\mathcal{O}(\mathrm{Aut}^{+}(\Lambda))$ or $ \cl O(\Lambda)$ for simplicity for the associated Hopf $*$-algebra.


\section{The hypergraph isomorphism game}\label{s_hig}

Recall that an \textit{isomorphism} from hypergraph $\Lambda_{1}$ to $\Lambda_{2}$ is a bijective function $f: V_{1}\rightarrow V_{2}$ such that $f^{-1}(e_{2})$ is an edge in $E_{1}$ if and only if $e_{2}$ is an edge of $E_{2}$. In this case, as $|V_{1}| = |V_{2}|$ we often identify them and set $V := V_{1} = V_{2}$. The existence of such a bijection $f: V_{1}\rightarrow V_{2}$ is equivalent to the existence of a bijection $g: E_{1}\rightarrow E_{2}$ for which $v_{1} \in e_{1}$ if and only if $f(v_{1}) \in g(e_{1})$, for all $v_{1} \in V_{1}$ and $e_{1} \in E_{1}$. When $\Lambda_{1}$ and $\Lambda_{2}$ are isomorphic (via some pair $(f, g)$), we write $\Lambda_{1} \cong \Lambda_{2}$.

\begin{remark} \label{r:classintertwin} \rm  
Let $\Lambda_{i} = (V_{i}, E_{i}), i = 1, 2$ be hypergraphs. Then it is easily verified that there is a bijective correspondence between the set of isomorphisms between $\Lambda_1$ and $ \Lambda_2$ and the set of all pairs $(P_{V}, P_{E})$ where $P_{V}$ (respectively, $P_{E}$) is a $V_1\times V_2$ permutation matrix (respectively, $E_1\times E_2$ permutation matrix) such that 
\begin{gather*}\label{permutation_rel}
    A_{\Lambda_{1}}P_{E} = P_{V}A_{\Lambda_{2}}.
\end{gather*}
Note that, for finite sets $X, Y$, an $X\times Y$ permutation matrix is a matrix in $M_{X, Y}(\{0, 1\})$ such that summing over the entries of each row and each column yields one. 
\end{remark}

\subsection{Game description}
Given two hypergraphs $\Lambda_1=(V_1,E_1)$ and $\Lambda_2=(V_2,E_2)$,
set
\[
V := V_1 \,\sqcup\, V_2, 
\qquad 
E := E_1 \,\sqcup\, E_2,
\]
and let the question and answer sets coincide: $X=A:=V\sqcup E$.

In each round, the verifier chooses $(x,y)\in X\times X$ and sends $x$ to Alice,
$y$ to Bob. They respond with $a,b\in A$. The winning rules for the hypergraph isomorphism game $\mathrm{HypIso}(\Lambda_1,\Lambda_2)$ are:
\begin{enumerate}
  \item[(i)] \textit{Type/swap.} The answers from each player must be of the same type as their questions, but from the opposite hypergraph. That is, if $x\in V_i$ then $a\in V_{3-i}$; if $x\in E_i$ then $a\in E_{3-i}$,
  and similarly for Bob for $i = 1, 2$. 

  \item[(ii)] \textit{Incidence preservation.}  The pairs $(x,a)$ and $ (y,b)$ satisfy the same incidence relations.
  If $x\in V_i$ and $y\in E_i$, then
  \[
     x\in y \;\Longleftrightarrow\; a\in b,
  \]
  where $a\in V_{3-i}$ and $b\in E_{3-i}$, $i = 1, 2$.

  \item[(iii)] \textit{Vertex--vertex relation preservation.}
  For $i = 1, 2$ and $u,u'\in V_i$, define
  \[
    \mathrm{rel}_V^i(u,u')\in\{\text{equal},\ \text{adjacent},\ \text{distinct non-adjacent}\},
  \]
  where ``adjacent'' means $\exists\,e\in E_i:\ u,u'\in e$.
  Whenever both questions are vertices (possibly from different hypergraphs), collect the two
  $V_1$-vertices among $\{x,y,a,b\}$ (call them $v_A,v_B$) and the two $V_2$-vertices
  among $\{x,y,a,b\}$ (call them $w_A,w_B$), and require
  \[
     \mathrm{rel}_V^1(v_A,v_B) \;=\; \mathrm{rel}_V^2(w_A,w_B).
  \]

  \item[(iv)] \textit{Edge--edge relation preservation.}
  For $i\in\{1,2\}$ and $e,e'\in E_i$, define
  \[
    \mathrm{rel}_E^i(e,e')\in\{\text{equal},\ \text{intersecting},\ \text{disjoint}\},
  \]
  where ``intersecting'' means $e\cap e'\neq\varnothing$.
  Whenever both questions are edges (possibly from different hypergraphs), collect the two
  $E_1$-edges among $\{x,y,a,b\}$ (call them $e_A,e_B$) and the two $E_2$-edges
  among $\{x,y,a,b\}$ (call them $f_A,f_B$), and require
  \[
     \mathrm{rel}_E^1(e_A,e_B) \;=\; \mathrm{rel}_E^2(f_A,f_B).
  \]
\end{enumerate}

\noindent
If the verifier asks the same question to both players, then by (i)–(iii)/(iv) the only accepted response requires Alice and Bob to output the same answer (on the opposite side); hence the game is synchronous and in fact bisynchronous.

\medskip

Let $\Lambda_1,\Lambda_2$ be hypergraphs. We say that $\mathrm{HypIso}(\Lambda_1,\Lambda_2)$
\emph{has a perfect $A^*$-strategy} if the $*$-algebra $\cl A(\mathrm{HypIso}(\Lambda_1,\Lambda_2))\neq 0$; and it
\emph{has a perfect $\mathrm{C}^*$-strategy} if
$\cl A(\mathrm{HypIso}(\Lambda_1,\Lambda_2))$ admits a unital $*$-representation into
$\cl B(H)$ for some Hilbert space $H$.

\begin{definition}\rm
Given two hypergraphs $\Lambda_1$ and $\Lambda_2$, we write
\[
\Lambda_{1} \cong_{\mathrm{loc}} \Lambda_{2},\quad
\Lambda_{1} \cong_{\mathrm{q}} \Lambda_{2},\quad
\Lambda_{1} \cong_{\mathrm{qa}} \Lambda_{2},\quad
\Lambda_{1} \cong_{\mathrm{qc}} \Lambda_{2}, \quad
\Lambda_1 \cong_{C^*} \Lambda_2, \quad \Lambda_1 \cong_{A^*} \Lambda_2,
\]
to denote that the $\operatorname{HypIso}(\Lambda_1,\Lambda_2)$ admits a perfect
\emph{local (classical)}, \emph{quantum}, \emph{quantum approximate}, \emph{quantum commuting}
strategy, \emph{has a perfect $\mathrm{C}^*$-strategy}, or \emph{has a perfect $A^*$-strategy} respectively.
\end{definition}


\subsection{Presentation of the game algebra}
In this subsection we will show that the game algebra is the universal object that encodes quantum isomorphisms of hypergraphs.

The following is an analogue of \cite[Proposition 2.1.3]{schmidt_one}, or \cite[Proposition 2.29]{frs} in the case of two hypergraphs; we include a proof for the convenience of the reader. 

\begin{lemma}\label{l_inclusion_rel}
Let $\Lambda_1 = (V_1, E_1)$ and $\Lambda_2 = (V_2, E_2)$ be hypergraphs with incidence matrices
$A_{\Lambda_1},A_{\Lambda_2}$.
Let $P_{V}=(p_{v_1,v_2})_{v_1\in V_1,\ v_2\in V_2}$ and $P_{E}=(p_{e_1,e_2})_{e_1\in E_1,\ e_2\in E_2}$ be magic unitaries.
Then the following are equivalent:
\begin{enumerate}
    \item[(i)] $A_{\Lambda_1} P_{E} \;=\; P_{V} A_{\Lambda_2}$;
    \item[(ii)] $ \sum_{\substack{e_1 \in E_1\\ v_1 \in e_1}} p_{e_1,e_2}
  =
  \sum_{\substack{v_2 \in V_2\\ v_2 \in e_2}} p_{v_1,v_2}
  \;\; {for \; all}\; v_1 \in V_1,\ e_2 \in E_2$;
 \item[(iii)]  $p_{v_1,v_2}\,p_{e_1,e_2} \;=\; p_{e_1,e_2}\,p_{v_1,v_2} = 0$ whenever 
 $ v_1\notin e_1$ and  $ v_2\in e_2 $ or $ v_1\in e_1$ and $  v_2\notin e_2$.
\end{enumerate}
\end{lemma}

\begin{proof}

[(i) $\Leftrightarrow$ (ii)] is straightforward equating the entries.

 [(ii) $\Rightarrow$ (iii)]  For all $v_1\in V_1$, $e_2\in E_2$,
\[
\sum_{\substack{e_1\in E_1\\ v_1\in e_1}} p_{e_1,e_2}
\;=\;
\sum_{\substack{v_2\in V_2\\ v_2\in e_2}} p_{v_1,v_2}.
\]
If $v_1\notin e_1$ and $v_2\in e_2$, then
\[
p_{v_1,v_2}\,p_{e_1,e_2}
= p_{v_1,v_2}\Big(\sum_{\substack{w_2\in V_2\\ w_2\in e_2}} p_{v_1,w_2}\Big)p_{e_1,e_2}
= p_{v_1,v_2}\Big(\sum_{\substack{e_1'\in E_1\\ v_1\in e_1'}} p_{e_1',e_2}\Big)p_{e_1,e_2}
= 0,
\]
since $e_1\notin\{e_1':v_1\in e_1'\}$ and column orthogonality gives $p_{e_1',e_2}p_{e_1,e_2}=0$ for $e_1'\neq e_1$.
Taking adjoints yields $p_{e_1,e_2}p_{v_1,v_2}=0$. The case $v_1\in e_1$, $v_2\notin e_2$ is analogous.

[(iii) $\Rightarrow$ (ii)]  Fix $v_1\in V_1$, $e_2\in E_2$ and set
\[
D_{v_1,e_2}:=\sum_{\substack{e_1\in E_1\\ v_1\in e_1}} p_{e_1,e_2}
\;-\;
\sum_{\substack{v_2\in V_2\\ v_2\in e_2}} p_{v_1,v_2}.
\]
We claim $D_{v_1,e_2}=0$. For any $v_2\in V_2$,
\[
p_{v_1,v_2}\,D_{v_1,e_2}
= \sum_{\substack{e_1\in E_1\\ v_1\in e_1}} p_{v_1,v_2}\,p_{e_1,e_2}
 - \sum_{\substack{w_2\in V_2\\ w_2\in e_2}} p_{v_1,v_2}\,p_{v_1,w_2}.
\]
If $v_2\notin e_2$, the second sum vanishes by row orthogonality, while each term
$p_{v_1,v_2} p_{e_1,e_2}$ in the first sum vanishes by the hypothesis (since $v_1\in e_1$ and $v_2\notin e_2$).
Hence $p_{v_1,v_2}D_{v_1,e_2}=0$.

If $v_2\in e_2$, then by the hypothesis $p_{v_1,v_2} p_{e_1,e_2}=0$ for all $e_1$ with $v_1\notin e_1$.
Thus
\[
p_{v_1,v_2}\sum_{\substack{e_1\in E_1\\ v_1\in e_1}} p_{e_1,e_2}
= p_{v_1,v_2}\sum_{e_1\in E_1} p_{e_1,e_2}
= p_{v_1,v_2}\cdot 1
= p_{v_1,v_2},
\]
using the column partition $\sum_{e_1} p_{e_1,e_2}=1$. Also
$\sum_{w_2\in e_2} p_{v_1,v_2} p_{v_1,w_2}=p_{v_1,v_2}$ by row orthogonality.
Hence $p_{v_1,v_2}D_{v_1,e_2}=p_{v_1,v_2}-p_{v_1,v_2}=0$.

Therefore $p_{v_1,v_2}D_{v_1,e_2}=0$ for all $v_2$, and summing over $v_2$ (using the row partition $\sum_{v_2}p_{v_1,v_2}=1$) gives $D_{v_1,e_2}=0$. This is the entrywise form of $A_{\Lambda_1}P_{E}=P_{V}A_{\Lambda_2}$.
\end{proof}

\begin{proposition}
\label{prop:game-alg-presentation}
Let $\Lambda_1, \Lambda_2$ be two hypergraphs. Then $\cl A(\mathrm{HypIso}(\Lambda_1,\Lambda_2))$ is the
 unital $*$-algebra generated by elements
$\{p_{v_1,v_2}\}_{v_1\in V_1,v_2\in V_2}$ and $\{p_{e_1,e_2}\}_{e_1\in E_1,e_2\in E_2}$ such that:
\begin{itemize}
  \item[(i)] $P_{V}:=(p_{v_1,v_2})_{v_1\in V_1,\ v_2\in V_2}$ and $P_{E}:=(p_{e_1,e_2})_{e_1\in E_1,\ e_2\in E_2}$ are magic unitaries;
  \item[(ii)] the incidence intertwining holds:
  \[
     A_{\Lambda_1}\,P_{E} \;=\; P_{V}\,A_{\Lambda_2}.
  \]
\end{itemize}
\end{proposition}

\begin{proof}

By the type/swap and synchronicity rule, for $x\in V_i$ (resp.\ $x\in E_i$) the only potentially
non-zero $p_{x,a}$ are those with $a\in V_{3-i}$ (resp.\ $a\in E_{3-i}$). Indeed, let $v_i,v_i' \in V_{i}$; then $\lambda(v_i,b,v_i',y)=0$ for any $b,y$, so that 
\[
p_{v_i,v_i'} = p_{v_i,v_i'}\Big(\sum_{b}p_{y,b}\Big)= \sum_{b}p_{v_i,v_i'}p_{y,b}=0.
\]
For $v_i\in V_{i}$ and $e_i\in E_{i}$ note that 
$\lambda(v_i,b,e_i,y)=0$ for any $b,y$ since answers must match the question type. Hence, a similar argument shows the rest of them. Thus $(p_{x,a})_{x,a\in X}$ consists of matrices: 
\[
P_{V} := (p_{v_1,v_2})_{v_1\in V_1,\ v_2\in V_2},\qquad
P_{E} := (p_{e_1,e_2})_{e_1\in E_1,\ e_2\in E_2}.
\]

We will show that they are magic unitaries. We first treat $P_{V}=(p_{v_1,v_2})_{v_1\in V_1,\ v_2\in V_2}$. We note that in the game $*$-algebra, each $p_{x,y}$ is a selfadjoint idempotent; thus it suffices to show that the rows and columns are orthogonal and add to $1$. Fix $v_1\in V_1$; synchronicity for the question pair $(v_1,v_1)$ forces that two different vertex answers on the $V_2$ side are mutually exclusive, hence
\[
p_{v_1,v_2}\,p_{v_1,v_2'}=0 \quad \text{for } v_2\neq v_2'.
\]
By the row partition and previous arguments, $\sum_{v_{2} \in V_{2}}p_{v_{1}, v_{2}} = 1$.

\smallskip
Using the row partition, we have
\[
p_{v_1,v_2}
= p_{v_1,v_2}\Big(\sum_{v_1'\in V_1} p_{v_2,v_1'}\Big)
= \sum_{v_1'\in V_1} p_{v_1,v_2}\,p_{v_2,v_1'}.
\]
By the relationship preservation rule
\[
\lambda\big(v_1,v_2, v_2,v_1'\big)=0 \quad\text{unless}\quad v_1'=v_1,
\]
so that
\[
p_{v_1,v_2}\,p_{v_2,v_1'}=0 \quad \text{for } v_1'\neq v_1.
\]
Therefore the sum collapses to
\[
p_{v_1,v_2} \;=\; p_{v_1,v_2}\,p_{v_2,v_1}.
\]
By symmetry (swap the roles of $v_1$ and $v_2$ in the previous argument) we also get
\[
p_{v_2,v_1} \;=\; p_{v_2,v_1}\,p_{v_1,v_2}.
\]
Hence 
\[
p_{v_1,v_2}=(p_{v_1,v_2})^{*}= (p_{v_1,v_2}\,p_{v_2,v_1})^*=  p_{v_2,v_1}\,p_{v_1,v_2}=p_{v_2,v_1}.
\]
In particular, we may freely identify the transpose block with $P_{V}$.

Fix $v_2\in V_2$. Using bisynchronicity,
\[
p_{v_1,v_2}\,p_{w_1,v_2}
= p_{v_2,v_1}\,p_{v_2,w_1}=0 \quad (v_1\neq w_1),
\]
by row orthogonality again,
\[
\sum_{v_1\in V_1} p_{v_1,v_2}
= \sum_{v_1\in V_1} p_{v_2,v_1} = 1.
\]
Thus $P_{V}$ is a magic unitary. The same arguments apply verbatim to $P_{E}=(p_{e_1,e_2})_{e_1\in E_1,\ e_2\in E_2}$. 
Therefore $P_{E}$ is also a magic unitary.

By Lemma \ref{l_inclusion_rel}, the intertwining relation is equivalent to showing that 
for all $v_1 \in V_1, e_1 \in E_1, v_2 \in V_2, e_2 \in E_2$,
\[
p_{v_1,v_2}\,p_{e_1,e_2} = p_{e_1,e_2}\,p_{v_1,v_2} = 0,
\]
whenever either $v_1 \notin e_1$ and $v_2 \in e_2$ or $v_1 \in e_1$ and $v_2 \notin e_2$,
which is forced by the incidence preservation rule.

Conversely, let $\mathcal B$ be the unital $*$-algebra generated
by two magic unitaries $P_{V}=(p_{v_1,v_2})$ and $P_{E}=(p_{e_1,e_2})$ satisfying
$A_{\Lambda_1}P_{E}=P_{V}A_{\Lambda_2}$. Define a family
$\{q_{x,a}\}_{x,a\in X}$ by
\[
q_{x,a}=
\begin{cases}
p_{v_1,v_2} & x=v_1\in V_1,\ a=v_2\in V_2,\\
p_{v_2,v_1} & x=v_2\in V_2,\ a=v_1\in V_1,\\
p_{e_1,e_2} & x=e_1\in E_1,\ a=e_2\in E_2,\\
p_{e_2,e_1} & x=e_2\in E_2,\ a=e_1\in E_1,\\
0 & \text{(forbidden by type)}.
\end{cases}
\]
The generators $ q_{x,a}$ are selfadjoint idempotents that sum to the identity over $ a$, since $P_{V}$, $P_{E}$ are magic unitaries. We have to check that $ q_{x,a} q_{y,b}=0 $ whenever $ \lambda(x,y,a,b)=0$.
By Lemma~\ref{l_inclusion_rel}, incidence preservation follows  from the intertwining relation. Also, the type/swap rule is immediate since we set zero to all the relevant blocks. Hence it suffices to check the vertex-vertex and edge-edge relation preservation and in particular the adjacency/intersecting relation since equal to equal is taken care of by row/column orthogonality.

From the intertwining relations 
it follows that
\begin{equation}\label{eq:vertex-intertwiner}
A_{\Lambda_1}A_{\Lambda_{1}}^{\top}\,P_{V} \;=\; P_{V}\,A_{\Lambda_2}A_{\Lambda_2}^{\top}.
\end{equation}
Indeed: as $P_{V}A_{\Lambda_{2}} = A_{\Lambda_{1}}P_{E},$ then $P_{E}^{*}A_{\Lambda_{1}}^{\top} = A_{\Lambda_{2}}^{\top}P_{V}^{*}$. Multiplying both sides on the right by $P_{V}$ yields $P_{E}^{*}A_{\Lambda_{1}}^{\top}P_{V} = A_{\Lambda_{2}}^{\top}$, and thus $P_{E}A_{\Lambda_{2}}^{\top} = A_{\Lambda_{1}}^{\top}P_{V}$; this is  establishes (\ref{eq:vertex-intertwiner}).
Write $C_1:=A_{\Lambda_1}A_{\Lambda_1}^{\top}$ and
$C_2:=A_{\Lambda_2}A_{\Lambda_2}^{\top}$, so that
\[
C_i(v_i,w_i) \;=\; \sum_{e_i\in E_i} (A_{\Lambda_i})_{v_i,e_i}\,(A_{\Lambda_i})_{w_i,e_i},
\]
and hence $C_i(v_i,w_i)>0$ if and only if $v_i,w_i$ share a hyperedge in $\Lambda_i$ (meaning $C_i(v_i,w_i)$ is the number of edges containing both $v_i,w_i$).

It suffices to consider adjacency mismatches. By \eqref{eq:vertex-intertwiner},
\begin{equation}\label{eq:entrywise-V}
\sum_{w_1:\,C_1(v_1,w_1)>0} p_{w_1,\,v_2}
\;=\;
\sum_{v_2':\,C_2(v_2',v_2)>0} p_{v_1,\,v_2'}
\qquad(v_1,v_2)\in V_1\times V_2.
\end{equation}
Assume $C_1(v_1,w_1)>0$ and $C_2(v_2,w_2)=0$.
Multiply \eqref{eq:entrywise-V} (with fixed $(v_1,w_2)$) on the left by $p_{v_1,v_2}$
and on the right by $p_{w_1,w_2}$:
\[
p_{v_1,v_2}\Big(\sum_{x_1:\,C_1(v_1,x_1)>0} p_{x_1,w_2}\Big)p_{w_1,w_2}
\;=\;
p_{v_1,v_2}\Big(\sum_{v_2':\,C_2(v_2',w_2)>0} p_{v_1,v_2'}\Big)p_{w_1,w_2}.
\]
On the right, $C_2(v_2,w_2)=0$ implies $v_2\notin\{v_2':C_2(v_2',w_2)>0\}$, hence
$p_{v_1,v_2}p_{v_1,v_2'}=0$ for every summand by row orthogonality, so the right-hand side is $0$.
On the left, column orthogonality annihilates all terms except $x_1=w_1$, giving
$p_{v_1,v_2}p_{w_1,w_2}$. Thus $p_{v_1,v_2}p_{w_1,w_2}=0$. On the other hand, assume $C_1(v_1,w_1)=0$ and $C_2(v_2,w_2)>0$. The same calculation yields $0$ on the left-hand side
(because $w_1\notin\{x_1:C_1(v_1,x_1)>0\}$), while the right-hand side reduces to $p_{v_1,v_2}p_{w_1,w_2}$
(since $v_2\in\{v_2':C_2(v_2',w_2)>0\}$), hence again $p_{v_1,v_2}p_{w_1,w_2}=0$. Taking adjoints gives $p_{w_1,w_2}p_{v_1,v_2}=0$ in both cases.

From the intertwining we also obtain
\begin{equation}\label{eq:edge-int}
P_{E}\,A_{\Lambda_2}^{\top}A_{\Lambda_2} \;=\; A_{\Lambda_1}^{\top}A_{\Lambda_1}\,P_{E}.
\end{equation}
With  $D_1:=A_{\Lambda_1}^{\top}A_{\Lambda_1}$ and
$D_2:=A_{\Lambda_2}^{\top}A_{\Lambda_2}$, we have
\[
D_{i}(e_i,f_i) \;=\; \sum_{v_i\in V_{i}} (A_{\Lambda_{i}})_{v_i,e_i}\,(A_{\Lambda_{i}})_{v_i,f_i}
= |\,e_i\cap f_i\,|,
\]
so $D_i(e_i,f_i)>0$ if and only if $e_i,f_i$ intersect, and $D_i(e_i,f_i)=0$ if and only if $e_i,f_i$ are disjoint. Rewrite \eqref{eq:edge-int} entrywise: for every $(e_1,e_2)\in E_1\times E_2$,
\begin{equation}\label{eq:edge-entry}
\sum_{e_2' :\,D_2(e_2,e_2')>0} p_{e_1,e_2'}
\;=\;
\sum_{e_1' :\,D_1(e_1',e_1)>0} p_{e_1',e_2}.
\end{equation} 
Assume $D_1(e_1,e_1')>0$ and $D_2(e_2,e_2')=0$. Multiply \eqref{eq:edge-entry}
on the left by $p_{e_1',e_2'}$ and on the right by $p_{e_1,e_2}$:
\[
p_{e_1',e_2'}\Big(\sum_{e_2'':\,D_2(e_2,e_2'')>0} p_{e_1,e_2''}\Big)p_{e_1,e_2}
\;=\;
p_{e_1',e_2'}\Big(\sum_{e_1'':\,D_1(e_1'',e_1)>0} p_{e_1'',e_2}\Big)p_{e_1,e_2}.
\]
On the left-hand side, $e_2'\notin\{e_2'':D_2(e_2,e_2'')>0\}$, so by column orthogonality every term is $0$.
On the right-hand side, since $e_1'\in\{e_1'':D_1(e_1'',e_1)>0\}$, column orthogonality annihilates all terms
except $e_1''=e_1'$, leaving $p_{e_1',e_2'}p_{e_1,e_2}$. Hence $p_{e_1',e_2'}p_{e_1,e_2}=0$. Now assume $D_1(e_1,e_1')=0$ and $D_2(e_2,e_2')>0$. The same computation gives
$0$ on the right-hand side (since $e_1'\notin\{e_1'':D_1(e_1'',e_1)>0\}$), while the left-hand side reduces to
$p_{e_1',e_2'}p_{e_1,e_2}$ (as $e_2'\in\{e_2'':D_2(e_2,e_2'')>0\}$). Thus $p_{e_1',e_2'}p_{e_1,e_2}=0$.
Taking adjoints yields $p_{e_1,e_2}p_{e_1',e_2'}=0$.

We conclude that  $\mathcal A(\mathrm{HypIso}(\Lambda_1,\Lambda_2))$ is the 
unital $*$-algebra generated by $\{p_{v_1,v_2}\}$ and $\{p_{e_1,e_2}\}$ with
$P_{V},P_{E}$ magic unitaries and $A_{\Lambda_1}P_{E}=P_{V}A_{\Lambda_2}$, as
claimed.
\end{proof}

\begin{remark} \label{rmk_samesizes} \rm 
We note that using the characterization of the game $*$-algebra obtained in Proposition \ref{prop:game-alg-presentation} we have that $ |V_1|=|V_2|$ and $ |E_1|=|E_2|$, whenever $ \cl A({\rm HypIso}(\Lambda_1, \Lambda_2)) \neq 0$. Indeed, note that since the matrices $(p_{v_1,v_2})_{v_{1}, v_{2}}$ and $(p_{e_1,e_2})_{e_{1}, e_{2}}$ are magic unitaries, $ |V_1|=\sum_{v_{1}} \sum_{v_{2}} p_{v_1,v_2}=\sum_{v_{2}} \sum_{v_{1}} p_{v_1,v_2}= |V_2| $ (and a similar relation holds for the edges). Thus without loss of generality we may assume that the hypergraphs considered in this work have the same vertex and edge cardinalities.
\end{remark}

Using the characterization obtained by results in \cite{hmps, kps} (see Theorem \ref{th:synchro_charact}) and Remark \ref{r:classintertwin} we have:

\begin{theorem} \label{prop:quantum_isomorphisms}
    Given hypergraphs $ \Lambda_1, \Lambda_2$, the following hold
    \begin{enumerate}
        \item[(i)] $ \Lambda_1 \cong_{\rm loc} \Lambda_2$ if and only if $ \Lambda_1 \cong \Lambda_2$,
        \item[(ii)] $ \Lambda_1 \cong_{\rm q} \Lambda_2$ if and only if there exist finite dimensional Hilbert space $H$ and quantum permutation matrices $P_{V} = (p_{v_1,v_2})_{v_1 \in V_{1}, v_2 \in V_{2}}$ and $P_{E} = (p_{e_1,e_2})_{e_1 \in E_{1}, e_2 \in E_{2}}$ with entries in $\cl B(H)$ such that $ A_{\Lambda_{1}}P_{E} = P_{V}A_{\Lambda_{2}}$;
        \item[(iii)] $ \Lambda_1 \cong_{\rm qa} \Lambda_2$ if and only if  quantum permutation matrices $P_{V}=(p_{v_1,v_2})_{v_1 \in V_{1}, v_2 \in V_{2}}$ and $ P_{E}= (p_{e_1,e_2})_{e_1 \in E_{1}, e_2 \in E_{2}}$ with entries  in an ultrapower of the hyperfinite $ II_1$-factor $ \cl R^{\omega}$  and  such that  $A_{\Lambda_1}P_{E}= P_{V} A_{\Lambda_2}$;
        \item[(iv)] $ \Lambda_1 \cong_{\rm qc} \Lambda_2$ if and only if there exists a unital $C^{*}$-algebra $A$ equipped with a  faithful tracial state and  quantum permutation matrices $P_{V}=(p_{v_1,v_2})_{v_1 \in V_{1}, v_2 \in V_{2}}$ and $ P_{E}= (p_{e_1,e_2})_{e_1 \in E_{1}, e_2 \in E_{2}}$ with entries  in $ A$  and  such that  $A_{\Lambda_1}P_{E}= P_{V} A_{\Lambda_2}$;
        \item[(v)] $ \Lambda_1 \cong_{C^*} \Lambda_2$ if and only if there exist a Hilbert space $H$ and quantum permutation matrices $P_{V} = (p_{v_1,v_2})_{v_1 \in V_{1}, v_2 \in V_{2}}$ and $P_{E} = (p_{e_1,e_2})_{e_1 \in E_{1}, e_2 \in E_{2}}$ with entries in $\cl B(H)$ such that $ 
        A_{\Lambda_{1}}P_{E} = P_{V}A_{\Lambda_{2}}$. 
    \end{enumerate}
\end{theorem}

\subsection{Connection to the graph isomorphism game}
We conclude this section by relating the hypergraph isomorphism game algebra to the known graph isomorphism setting \cite{AMRSVV} when the hypergraphs are simple graphs. In Proposition~\ref{prop:GI-vs-QAut-graph}, we will show that if $G_1,G_2$ are simple graphs then
$\mathcal A(\mathrm{HypIso}(G_1,G_2))$ admits a vertex–only presentation with the
adjacency intertwiner and the edgewise commutation relations. Our approach follows the one in~\cite[Section 4]{frs}, extended here to pairs of hypergraphs and $*$-algebras.

Let $G_i=(V_i,E_i)$ be simple graphs for $i = 1, 2$. We write $A^{(G_i)}\in M_{V_i}(\{0,1\})$ for the adjacency matrix and $A_i\in M_{V_i\times E_i}(\{0,1\})$ for the vertex--edge incidence matrix.
For simple graphs one has the well-known identity
\begin{equation}\label{eq:AAstar-adj-degree}
 A_iA_i^{*} \,=\, A^{(G_i)} + D_i, \qquad D_i:=\operatorname{diag}\big(\deg_i(v)\big)_{v\in V_i},
\end{equation}
\noindent where ${\rm deg}_{i}(v)$ denoted the number of vertices adjacent to $v$ in $G_{i}$, $i = 1, 2$.

Let $\Lambda=(V,E)$ be a hypergraph.  
For a vertex $v\in V$, the \emph{degree} of $v$ is the number of hyperedges containing it:
\( 
\deg(v):=|\{\,e\in E : v\in e\,\}|.
\)
When necessary, we use ${\rm deg}_{i}(v)$ to denote the degree of $v$ inside hypergraph $\Lambda_{i}$, $i = 1, 2$. Note that when $\Lambda = (V, E)$ is a graph, ${\rm deg}(v)$ coincides with the number of vertices adjacent to $v$ in $\Lambda$, for any $v \in V$.

\begin{lemma}\label{lem:degree-intertwine-hyp}
Let $\Lambda_1=(V_1,E_1)$ and $\Lambda_2=(V_2,E_2)$ be hypergraphs, and suppose
$P_V,P_E$ satisfy the intertwining relation
\[
A_{\Lambda_1}\,P_E \;=\; P_V\,A_{\Lambda_2}.
\]
If $D_i=\operatorname{diag}\big(\deg(v_i)\big)_{v_i\in V_i}$  denotes the  degree matrix of $\Lambda_i$, then
\begin{equation}\label{eq:degree-blocks-hyp}
D_1\,P_V \,=\, P_V\,D_2,
\end{equation}
for all $v_1\in V_1$, $v_2\in V_2$. In particular, $p_{v_1,v_2}=0$ unless $\deg_1(v_1)=\deg_2(v_2)$.
\end{lemma}

\begin{proof}
As $P_E$ is a magic unitary, $P_E\,\mathbf 1_{E_2}=\mathbf 1_{E_1}$ where ${\bf 1}_{E_{i}}\in \bb C^{E_{i}}, i = 1, 2$, is the all-one column vector.
Multiplying $P_{V}A_{\Lambda_{2}}$ on the right by $\mathbf 1_{E_2}$ yields
\[
P_V\,A_{\Lambda_2}\,\mathbf 1_{E_2}\;=\; A_{\Lambda_1}\,\mathbf 1_{E_1}.
\]
This is $\deg_1 = P_V \deg_2$, where $\deg_i$ is the degree vector, which gives the matrix identity \eqref{eq:degree-blocks-hyp}. Indeed, it is easily checked that $A_{\Lambda_{i}}\mathbf{1}_{E_{i}} = {\rm deg}_{i}$, $i =1, 2$. As ${\rm deg}_{1} = P_{V}{\rm deg}_{2}$, we see in particular that for each $v_{k} \in V_{1}$,
\begin{gather}\label{eqn_deg_rel}
    {\rm deg}_{1}(v_{k}) = \sum\limits_{v_{j} \in V_{2}}p_{kj}{\rm deg}_{2}(v_{j}).
\end{gather}
Relation (\ref{eq:degree-blocks-hyp}) holds if and only if for each $v_{k} \in V_{1}, v_{j} \in V_{2}$ the relation 
\begin{gather*}
    {\rm deg}_{1}(v_{k})p_{kj} = p_{kj}{\rm deg}_{2}(v_{j})
\end{gather*}
\noindent holds; using (\ref{eqn_deg_rel}) and orthogonality, we have
\begin{gather*}
    {\rm deg}_{1}(v_{k})p_{kj} = \sum\limits_{v_{\ell} \in V_{2}}p_{k\ell}{\rm deg}_{2}(v_{\ell})p_{kj} = p_{kj}{\rm deg}_{2}(v_{j}),
\end{gather*}
\noindent establishing our claim.
\end{proof}

\begin{lemma}\label{lem:edge-from-vertex}
Let $G_i=(V_i,E_i)$ be simple graphs for $i = 1, 2$ and $\mathcal A(\mathrm{HypIso}(G_1,G_2))$ be the hypergraph game $*$-algebra.  If $v_1\sim_{G_1} w_1$ and $v_2\sim_{G_2} w_2$, then
\begin{equation}\label{eq:edge-commute}
 p_{v_1,v_2}\,p_{w_1,w_2} \;=\; p_{w_1,w_2}\,p_{v_1,v_2}.
\end{equation}
\end{lemma}

\begin{proof}
Let  $e_1=\{v_1,w_1\}\in E_1$ and $e_2=\{v_2,w_2\}\in E_2$. We will show first that 
\[
p_{e_1,e_2} = p_{v_1,v_2}\,p_{w_1,w_2} \;+\; p_{v_1,w_2}\,p_{w_1,v_2}.
\]
From the intertwining relation $A_{1}P_E=P_VA_{2}$ we have, for every $u\in V_1$ and $e_2=\{v_2,w_2\}$,
\[
\sum_{\substack{f_1\in E_1\\ u\in f_1}} p_{f_1,e_2} \;=\; p_{u,v_2}+p_{u,w_2}.
\]
Apply this with $u=v_1$ and $u=w_1$, then multiply the two equalities:
\[
\Bigg(\sum_{\substack{f_1\in E_1\\ v_1\in f_1}} p_{f_1,e_2}\Bigg)
\Bigg(\sum_{\substack{g_1\in E_1\\ w_1\in g_1}} p_{g_1,e_2}\Bigg)
= (p_{v_1,v_2}+p_{v_1,w_2})(p_{w_1,v_2}+p_{w_1,w_2}).
\]
In the left-hand side of the equation, column orthogonality in $P_E$ annihilate all terms except $f_1=g_1$, and since the graphs are simple the unique edge containing both $v_1,w_1$ is $e_1$. Hence the left-hand side is $p_{e_1,e_2}$. Expanding the right-hand side and using column orthogonality in $P_V$  gives $p_{v_1,v_2}p_{w_1,w_2} + p_{v_1,w_2}p_{w_1,v_2}$.

Next we show that
\[
p_{e_1,e_2}= p_{v_1,v_2}\,p_{w_1,w_2} \;+\; p_{w_1,v_2}\,p_{v_1,w_2}, 
\]
use the transposed intertwining relation $P_{E}A_{{2}}^{\top} = A_{{1}}^{\top}P_{V}$  at the
$(e_1,v)$–entry, yields for any $v\in V_2$:
\begin{equation*}\label{eq:edge-vertex-sum}
\sum_{\substack{f_2\in E_2\\ v\in f_2}} p_{e_1,f_2}
\;=\; p_{v_1,v}+p_{w_1,v}.
\end{equation*}
Apply this with $v=v_2$ and $v=w_2$, multiply the two identities  and use column orthogonality to obtain 
\[
p_{e_1,e_2}
\;=\;
p_{v_1,v_2}p_{w_1,w_2}\;+\;p_{w_1,v_2}p_{v_1,w_2}.
\]
Now the result follows from combining the expressions of $p_{e_1,e_2}$ and using its selfadjointness.
\end{proof}

\begin{proposition}\label{prop:GI-vs-QAut-graph}
Let $G_1,G_2$ be simple graphs with adjacency matrices $A^{(G_1)},A^{(G_2)}$ and incidence matrices $A_1,A_2$. The algebra $\mathcal A(\mathrm{HypIso}(G_1,G_2))$ is canonically isomorphic to the unital $*$--algebra generated by the entries of a magic unitary $U=(u_{v_1,v_2})_{v_1\in V_1,v_2\in V_2}$ such that
\begin{align*}
 A^{(G_1)}\,U &= U\,A^{(G_2)},
\end{align*}
and $u_{v_1,v_2} u_{w_1,w_2} = u_{w_1,w_2} u_{v_1,v_2}$  whenever  $ v_1\sim_{G_1} w_1 \text{ and } v_2\sim_{G_2} w_2.$
\end{proposition}
\begin{proof}  Using \eqref{eq:AAstar-adj-degree} and the intertwining relation we have
\[
 A_1A_1^{*}\,P_V = (A_1P_E)A_2^{*} = (P_VA_2)A_2^{*} = P_V\,A_2A_2^{*}.
\]
By \eqref{eq:AAstar-adj-degree}, $(A^{(G_1)}+D_1)P_V=P_V(A^{(G_2)}+D_2)$. Lemma~\ref{lem:degree-intertwine-hyp} gives $D_1P_V=P_VD_2$, hence
\[
 A^{(G_1)}P_V = P_VA^{(G_2)},
\]
which is the  adjacency intertwining relation with $U=P_V$. The commutation relations  hold by Lemma \ref{lem:edge-from-vertex}.

\smallskip
Conversely, suppose $\widetilde{P_V}:=U=(u_{v_1,v_2})$ is a magic unitary satisfying the adjacency intertwining relation and the commutation relations from the statement. Define
\begin{equation}\label{eq:edge-lift}
 \widetilde p_{v_1,v_2}:=u_{v_1,v_2},
 \qquad
 \widetilde p_{\{v_1,w_1\},\{v_2,w_2\}}
 := u_{v_1,v_2}\,u_{w_1,w_2}+u_{v_1,w_2}\,u_{w_1,v_2}.
\end{equation}
 By assumption, $u_{v_1,v_2}$ commutes with $u_{w_1,w_2}$ (and $u_{v_1,w_2}$ with $u_{w_1,v_2}$), so each summand in \eqref{eq:edge-lift} is a product of commuting projections, hence a projection. Row/column orthogonality of $U$ makes the two summands orthogonal, so every $\widetilde p_{e_1,e_2}$ is a projection. Fix $e_1=\{v_1,w_1\}\in E_1$. Then
\begin{align*}
\sum_{\{v_2,w_2\}\in E_2} \widetilde p_{\{v_1,w_1\},\{v_2,w_2\}}
 & = \sum_{v_2\sim w_2} \big(u_{v_1,v_2}u_{w_1,w_2}+u_{v_1,w_2}u_{w_1,v_2}\big)\\
 & = \sum_{v_2,w_2} A^{(G_2)}_{v_2,w_2}\,u_{v_1,v_2}u_{w_1,w_2}\\
 &= (UA^{(G_2)}U^{*})_{v_1,w_1}.    
\end{align*}
By assumption, $UA^{(G_2)}=A^{(G_1)}U$, hence $(UA^{(G_2)}U^{*})_{v_1,w_1}=A^{(G_1)}_{v_1,w_1}=1$ because $v_1\sim w_1$. The fact that columns sum to the identity follows  analogously. Orthogonality of distinct rows/columns follows from that of $U$ by routine expansions. Therefore $\widetilde P_E = (\widetilde p_{e_1,e_2} )$ is a magic unitary. For $v_1\in V_1$ and $e_2=\{v_2,w_2\}\in E_2$,
\begin{align*}
 (A_1\widetilde P_E)_{v_1,e_2}
 & = \sum_{w_1:\,v_1\sim w_1}\!\big(u_{v_1,v_2}u_{w_1,w_2}+u_{v_1,w_2}u_{w_1,v_2}\big)\\
 &= u_{v_1,v_2} \Big(\sum_{w_1:\,v_1\sim w_1} u_{w_1,w_2}\Big)+ u_{v_1,w_2}\Big(\sum_{w_1:\,v_1\sim w_1} u_{w_1,v_2}\Big).
\end{align*}
By expressing the adjacency intertwining relation coordinate-wise, we have \begin{gather*}
    \sum_{w_1:\,v_1\sim w_1}u_{w_1,w_2}=\sum_{z_2:\,z_2\sim w_2}u_{v_1,z_2},
\end{gather*}
and likewise for the other sum; by row orthogonality of $U$  the sums reduce to $u_{v_1,v_2}$ and $u_{v_1,w_2}$ respectively. Hence $(A_1\widetilde P_E)_{v_1,e_2}=u_{v_1,v_2}+u_{v_1,w_2}=(\widetilde P_VA_2)_{v_1,e_2}$, proving $A_1\widetilde P_E=\widetilde P_VA_2$. Thus $(\widetilde P_V,\widetilde P_E)$ satisfy the defining relations, yielding an $*$-isomorphism between the two $*$-algebras. 
\end{proof}

Now recalling \cite{bigalois}, where $\cl A(\mathrm{Iso}(G_1,G_2))$ is the unital $*$-algebra generated by the entries of a magic unitary $ U =(u_{x,y})_{x\in V_1,y\in V_2}$ such that $ A^{({G_1})}U = U A^{(G_2)}$, we immediately have:
\begin{proposition} \label{p:HIGquotientGI}
Let \(G_1=(V_1,E_1)\) and \(G_2=(V_2,E_2)\) be simple graphs. Then
\[
\mathcal A(\mathrm{HypIso}(G_1,G_2))
\;\cong\;
\mathcal A(\mathrm{Iso}(G_1,G_2))
     /\bigl\langle\, 
       u_{v_1,v_2}u_{w_1,w_2}-u_{w_1,w_2}u_{v_1,v_2}
       :\; v_1\sim_{G_1}w_1,\;
         v_2\sim_{G_2}w_2
     \,\bigr\rangle.
\]
\end{proposition}

We remark that if $ G_1=G_2=:G$ then $ \mathcal A(\mathrm{Iso}(G,G)) = \cl O( {\rm Aut^+_{Ban}}(G))$, the Hopf $*$-algebra of the quantum automorphism group of Banica \cite{banica2} and $\mathcal A(\mathrm{HypIso}(G,G))= \cl O({\rm Aut^{+}_{Bic}}(G))$, the Hopf $*$-algebra of the quantum automorphism group of Bichon  \cite{bichon02}. In particular ${\rm Aut^{+}_{Bic}}(G)$ is a quantum subgroup of ${\rm Aut^+_{Ban}}(G)$.

\begin{corollary} 
Let $G_i=(V_i,E_i)$, $i=1,2$ be simple graphs. For ${\rm t} \in \{\rm loc,q,qa,qc, {\rm C^*}, A^*\}$, if $\mathrm{HypIso}(G_1,G_2)$ has a perfect $\rm t$-strategy, then $\mathrm{Iso}(G_1,G_2)$ has a perfect $ \rm t$-strategy. 
\end{corollary}
\begin{proof}
    By Theorem \ref{th:synchro_charact}, the existence of a perfect strategy of $\mathrm{HypIso}(G_1,G_2)$ translates to a $*$-representation of its game algebra. Composing with the canonical quotient map coming from Proposition \ref{p:HIGquotientGI} and invoking Theorem \ref{th:synchro_charact} again yields a representation of $\cl A(\mathrm{Iso}(G_1,G_2))$ of the same type.
\end{proof}

We note, in particular, that by Proposition~\ref{p:HIGquotientGI}, the algebra 
\(\cl A(\mathrm{HypIso}(G_1,G_2))\) is obtained from 
\(\cl A(\mathrm{Iso}(G_1,G_2))\) by imposing additional relations on the same 
set of generators. Consequently, any perfect strategy \(p\) for 
\(\mathrm{HypIso}(G_1,G_2)\) is also a perfect strategy for 
\(\mathrm{Iso}(G_1,G_2)\). In this sense, 
\(\mathrm{HypIso}(G_1,G_2)\) may be regarded as an algebraic subgame of 
\(\mathrm{Iso}(G_1,G_2)\).

\smallskip

 Given graphs $G_i$, define the neighborhood hypergraphs 
\[
\cl N(G_i) := \bigl(V_i,\ \{\mathcal N_i(x)\}_{x\in V_i}\bigr),
\qquad
\mathcal N_i(x):=\{x'\in V_i:\ x'\sim_{G_i} x\},
\]
for $i=1,2$.
\begin{lemma} \label{p:chainofgames}
For simple graphs $G_i=(V_i,E_i)$, $i=1,2$, there is a surjective $*$-homomorphism 
\begin{equation*}\label{eq:chain_of_games}
 \cl A({\rm HypIso}(\cl N(G_1),\cl N(G_2))) \twoheadrightarrow \cl A(\mathrm{Iso}(G_1,G_2) ).
\end{equation*}
In particular, if $\mathrm{Iso}(G_1,G_2) $ has a perfect $\rm t$-strategy, then ${\rm HypIso}(\cl N(G_1),\cl N(G_2))$ has a perfect $\rm t$-strategy,  for ${\rm t} \in \{\rm loc,q,qa,qc, {\rm C^*}, A^*\}$.
\end{lemma}
\begin{proof}
Let $ E_i$ be the edge set of $ \cl N(G_i)$; then there exists a bijection $ \eta_i : V_i \to E_i$, $ \eta_i (x_i) = \cl N(x_i)$, $x_i \in V_i $ for $i = 1, 2$. This induces a permutation matrix $C_i \in M_{V_i\times E_i}$ such that $ A_{\cl N(G_i)} = A^{(G_i)} C_i$, where $A_{\cl N(G_i)}$ is the incidence matrix of the neighborhood hypergraph and $A^{(G_i)} $ is the adjacency matrix of the graph. Then, if $U$ is the magic unitary in $ \cl A({\rm Iso}(G_1,G_2))$ and if we set $ P_V':=U$ and $ P_E':= C_1^* U C_2$, we see that $ A_{\cl N(G_1)}P_E' = P_V'A_{\cl N(G_2)}$. By universality there exists a unital $*$-homomorphism $ \phi : \cl A( {\rm HypIso}(\cl N(G_1),\cl N(G_2))) \to \cl A({\rm Iso}(G_1,G_2)) $ mapping $ \phi(P_V) = P_V'$ and $ \phi(P_E)= P_E'$  that is moreover a surjection since $ U$ generates $\cl A({\rm Iso}(G_1,G_2))$. 
    
\end{proof}

\begin{theorem} \label{th:friendly_neighborhood}
    There exist hypergraphs $\Lambda_{1}$ and $\Lambda_{2}$ such that $\Lambda_{1} \cong_{\rm q} \Lambda_{2}$, while $\Lambda_{1} \not \cong_{\rm loc} \Lambda_{2}$.
\end{theorem}
\begin{proof}
    By \cite[Theorem 6.4]{AMRSVV}, there exist graphs $G_{i} = (V_{i}, E_{i}), i = 1, 2$ such that $G_{1} \cong_{\rm q} G_{2}$, while $G_{1} \not \cong G_{2}$.  Let $\cl N(G_i) = (V_{i}, \{\cl{N}(x)\}_{x \in V_{i}})$ be neighborhood hypergraphs over $V_{i}$. By  Lemma  \ref{p:chainofgames}, since $G_{1} \cong_{\rm q} G_{2}$ we have that $\cl N(G_1) \cong_{\rm q} \cl N(G_2)$.

Now, assume towards contradiction that $\cl{N}(G_{1}) \cong_{\rm loc} \cl{N}(G_{2})$.
 Using Remark \ref{r:classintertwin} and Theorem \ref{prop:quantum_isomorphisms} we may argue that $\cl{N}(G_{1}) \cong \cl{N}(G_{2})$. Now, use the arguments of \cite[Theorem 4.16(i)]{ht_one} to force $G_{1} \cong G_{2}$, contradicting our choice of graphs $G_{1}$ and $G_{2}$.  
\end{proof}

\begin{proposition} \label{p:strict_incl} 
    There exist graphs $ G_1, G_2$ such that 
    \[
     \mathrm{HypIso}(G_1,G_2) \neq \mathrm{Iso}(G_1,G_2).
    \]
    In particular, there is a perfect  $\rm q$-strategy for $\mathrm{Iso}(G_1,G_2)$ that is not a perfect $\rm q$-strategy for $\mathrm{HypIso}(G_1,G_2)$. 
\end{proposition}

The construction comes from \cite{AMRSVV}. We will first introduce the binary constraint system game first considered in \cite{cleve-mittal}.  

A \emph{binary constraint system} (BCS) \(\mathcal F=(M,b)\) consists of tuple $M = (x_{1}, \hdots, x_{n}, S_{1}, \hdots, S_{m})$ with Boolean variables
\(x_1,\dots,x_n\in\{0,1\}\), subsets $S_{\ell} \subseteq [n]$ for $\ell \in [m]$, and $b$ determines parity constraints indexed by \(\ell\in[m]\) of the form
\[
\bigoplus_{i\in S_\ell} x_i = b_\ell,
\qquad\ b_\ell\in\{0,1\},
\]
where \(\oplus\) denotes addition modulo~2.
A \emph{satisfying assignment} for constraint~\(\ell\) is a map
\(a:S_\ell\to\{0,1\}\) satisfying the parity equation; if $a_{i} = a(i)$ denotes the $i^{\rm th}$ assignment for $i \in S_{\ell}$, the set of all such assignments is
\[
\mathrm{Sat}_\ell(b_\ell):=\Bigl\{a:S_\ell\to\{0,1\}\ \Bigm|\ \bigoplus_{i\in S_\ell} a_i=b_\ell\Bigr\}.
\]

In the \emph{BCS game} for \(\mathcal F\), the referee chooses \((\ell,\ell')\in[m]\times[m]\) and sends \(\ell\) to Alice and \(\ell'\) to Bob.
Alice returns \(a\in\mathrm{Sat}_\ell(b_\ell)\), Bob returns \(a'\in\mathrm{Sat}_{\ell'}(b_{\ell'})\),
and the verifier accepts if and only if the two answers are \emph{consistent} on all shared variables:
\[
\forall i\in S_\ell\cap S_{\ell'}:\ a_i=a'_i.
\]
We call \(\mathcal F\) \emph{classically satisfiable} (resp.\ \emph{quantum satisfiable})
if there exists a perfect classical (resp.\ quantum) strategy.

The \emph{homogenization} of \(\mathcal F=(M,b)\) is the system
\(\mathcal F_0=(M,0)\) obtained by setting all right-hand sides \(b_\ell=0\); it shares the same supports \(S_\ell\).

Atserias et al. \cite{AMRSVV} associate to any BCS~\(\mathcal F\) a graph \(G(\mathcal F)\) whose:
\begin{itemize}
  \item[(i)] vertices are pairs \((\ell,a)\) with \(\ell\) a constraint index and \(a\in\mathrm{Sat}_\ell(b_\ell)\) a  satisfying assignment;
  \item[(ii)] \((\ell,a)\) and \((\ell',a')\) are adjacent if and only if they \emph{conflict} on some shared variable, i.e., \(\exists\,i\in S_\ell\cap S_{\ell'}\) with \(a_i\neq a'_i\).
\end{itemize}
Non-adjacency means consistency on all overlaps. For fixed \(M\), the family \(\{G(M,b):b\in\{0,1\}^m\}\) varies only in the parity pattern.

Let \(\mathcal F=(M,b)\) and its homogenization \(\mathcal F_0=(M,0)\).
By \cite[Theorem 6.3]{AMRSVV}
\[
\mathcal F \text{ has a perfect quantum strategy }
\iff
G(M,b)\text{ and }G(M,0)\text{ are quantum isomorphic.}
\]
In particular, \cite[Theorem~6.4]{AMRSVV} exhibits graphs that are quantum isomorphic but not isomorphic (as in Theorem \ref{th:friendly_neighborhood}). The specific pair comes from the Mermin--Peres magic square game and its homogenization; the perfect quantum strategy for the graph isomorphism game arises from the perfect quantum strategy for the magic square game. We will show that this particular perfect strategy is not a perfect \(\mathrm{q}\)-strategy for the hypergraph isomorphism game. By virtue of Proposition~\ref{prop:GI-vs-QAut-graph}, it suffices to show the non-commutation of two entries of the magic unitary corresponding to adjacent vertices.

\smallskip

A perfect quantum BCS strategy yields a magic unitary realizing a perfect quantum strategy for \(\mathrm{Iso}(G(M,b),G(M,0))\). Concretely, upon input \(x=(\ell,f)\) (from either graph), the player measures the projective measurement \(\{F_\ell^{(a)}\}_{a\in\mathrm{Sat}_\ell(b_\ell)}\), obtains outcome \(a\), and outputs
\( 
(\ell,\ f\oplus a).
\)
If \(x\in G(M,b)\), then \(f\oplus a\) has parity~\(0\) and belongs to \(G(M,0)\); conversely, if \(x\in G(M,0)\), the output lies in \(G(M,b)\). Hence, the strategy ensures that questions and answers do not lie in the same graph.

\smallskip

The Mermin--Peres magic square \cite{mermin90} encodes \(\mathcal F=(M,b)\) on variables \(x_{rc}\) (\(r,c\in\{1,2,3\}\)) with constraints
\begin{align*}
x_{r1}\oplus x_{r2}\oplus x_{r3}&=0 &&(r=1,2,3),\\
x_{1c}\oplus x_{2c}\oplus x_{3c}&=0 &&(c=1,2),\\
x_{13}\oplus x_{23}\oplus x_{33}&=1 &&(c=3),
\end{align*}
and homogenization \(\mathcal F_0=(M,0)\).
Each induces a \(24\)-vertex graph \(G(M,b)\) or \(G(M,0)\) (see \cite[Figures~1--2]{AMRSVV}).

\begin{proof}[Proof of Proposition~\ref{p:strict_incl}] 
Let $G_1,G_2$ be the graphs from \cite[Theorem~6.4]{AMRSVV} satisfying 
$G_1 \cong_{\mathrm q} G_2$.  
Explicitly, $G_1 = G(M,b)$ and $G_2 = G(M,0)$, 
where $\mathcal F = (M,b)$ is the Mermin--Peres magic-square BCS game 
and $\mathcal F_0 = (M,0)$ is its homogenization. 
By \cite[Theorem~6.3]{AMRSVV}, a perfect quantum strategy for the BCS game~$\mathcal F$ 
induces a perfect quantum strategy for the graph-isomorphism game 
$\mathrm{Iso}(G(M,b),G(M,0))$.

\smallskip

Label the vertices of $G_i$ by pairs $x=(\ell,f)$ and $y=(\ell',g)$, 
where $f$ and $g$ are local satisfying assignments of the corresponding constraints.
The projective measurement implementing the perfect quantum isomorphism strategy is 
the block matrix 
\[
u_{(\ell,f),(\ell',g)} = \delta_{\ell,\ell'}\, F_\ell^{(f\oplus g)},
\]
where each family $\{F_\ell^{(a)}\}_{a\in\mathrm{Sat}_\ell(b_\ell)}$ 
is the projective measurement used in the magic-square game.

\smallskip

The perfect magic-square strategy acts on the Hilbert space 
$H=\mathbb C^2 \otimes \mathbb C^2$ with a maximally entangled state 
and Pauli observables.  
Because the state is maximally entangled, as proved in \cite[Theorems~5.3--5.4 and Lemma~5.8]{AMRSVV}, 
this guarantees that the matrix $U=(u_{x,y})$ defined above 
is a magic unitary that intertwines the adjacency matrices of the two graphs, that is,
\[
A^{({G_1})} U = U A^{({G_2})}.
\]
Hence $U$ satisfies the adjacency-preservation relation that defines 
a quantum isomorphism between~$G_1$ and~$G_2$.

\smallskip
The perfect quantum strategy for $\mathcal F$ uses the Pauli matrices
\[
X=\begin{bmatrix}0&1\\[2pt]1&0\end{bmatrix},
\qquad
Z=\begin{bmatrix}1&0\\[2pt]0&-1\end{bmatrix},
\]
with projectors onto their $\pm1$ eigenspaces
\[
P_A^{(\pm)}=\tfrac12(I\pm A),
\]
where bit~$0$ corresponds to eigenvalue~$+1$ and bit~$1$ to~$-1$.  
The constraint measurements for the first row~$R_1$ (parity~$0$) 
and first column~$C_1$ (parity~$0$) use the observables
\[
R_1:\ \{X\!\otimes\! I,\ I\!\otimes\! X,\ X\!\otimes\! X\},
\qquad
C_1:\ \{X\!\otimes\! I,\ I\!\otimes\! Z,\ X\!\otimes\! Z\}.
\]
Hence the joint projectors for specific local assignments are
\[
F_{R_1}^{(0,0,0)} = P_X^{(+)}\!\otimes P_X^{(+)}, 
\qquad
F_{C_1}^{(0,0,0)} = P_X^{(+)}\!\otimes P_Z^{(+)}.
\]
We write
\[
P_X^{(+)}=\tfrac12
  \begin{bmatrix}1&1\\[2pt]1&1\end{bmatrix}, 
\qquad
P_Z^{(+)}=
  \begin{bmatrix}1&0\\[2pt]0&0\end{bmatrix}.
\]

\smallskip

Consider the first-row constraint $R_1$ on variables 
$(x_{11},x_{12},x_{13})$ and the first-column constraint 
$C_1$ on $(x_{11},x_{21},x_{31})$.
Choose assignments
\[
a_{R_1}=(0,0,0),\qquad a_{C_1}=(1,1,0),
\]
which disagree on the shared variable~$x_{11}$.  
Then
\[
v_1=(R_1,a_{R_1}), \quad w_1=(C_1,a_{C_1})\in V(G_1), \qquad
v_2=(R_1,a_{R_1}), \quad w_2=(C_1,a_{C_1})\in V(G_2)
\]
are adjacent in both graphs.

\smallskip
Inside the block of $U$ with $a_{R_1}\oplus a_{R_1}=(0,0,0)$ and  $a_{C_1}\oplus a_{C_1}=(0,0,0) $,
\[
u_{v_1,v_2}=F_{R_1}^{a_{R_1}\oplus a_{R_1}}=P_X^{(+)}\!\otimes P_X^{(+)},
\qquad
u_{w_1,w_2}=F_{C_1}^{a_{C_1}\oplus a_{C_1}}=P_X^{(+)}\!\otimes P_Z^{(+)}.
\]
Since $P_X^{(+)}$ and $P_Z^{(+)}$ do not commute, we have
\[
[\,u_{v_1,v_2},\,u_{w_1,w_2}\,]
 = P_X^{(+)}\!\otimes [P_X^{(+)},P_Z^{(+)}]\neq 0.
\]
Yet $v_1\sim_{G_1} w_1$ and $v_2\sim_{G_2} w_2$, 
so by Proposition~\ref{prop:GI-vs-QAut-graph} (the edgewise-commutation condition),
any perfect quantum strategy for the hypergraph-isomorphism game 
$\mathrm{HypIso}(G_1,G_2)$ must satisfy
\([u_{v_1,v_2},u_{w_1,w_2}]=0.\)

\smallskip

The explicit calculation above shows that the magic unitary $U$ implementing 
the perfect quantum graph-isomorphism strategy, obtained from the 
maximally entangled magic-square strategy, indeed satisfies the adjacency-preservation
relation $A_{G_1}U=UA_{G_2}$ (by \cite[Theorems~5.3--5.4 and Lemma~5.8]{AMRSVV}),
but fails the stronger edgewise-commutation property required for 
$\mathrm{HypIso}(G_1,G_2)$.  
Consequently, the perfect $\mathrm q$-strategy witnessing 
$G(M,b)\cong_{\mathrm q} G(M,0)$ is \emph{not} a perfect 
$\mathrm q$-strategy for $\mathrm{HypIso}(G(M,b),G(M,0))$,
and hence the two nonlocal games are distinct.
\end{proof}

Recall \cite{bigalois} that two synchronous games $\cl G_1$ and $ \cl G_2$ are called \emph{$*$-equivalent} if there exist unital $*$-homomorphisms $ \pi : \cl A(\cl G_1) \to \cl A(\cl G_2)$ and $ \rho : \cl A(\cl G_2) \to \cl A(\cl G_1)$. A particular instance of $*$-equivalent games is when they have isomorphic game algebras (see \cite{harris2021}). This relation allows to deduce the existence of perfect strategies for the one game from the existence of perfect strategies for the other game and vice versa.

Thus, in Proposition  \ref{p:strict_incl} we saw that  $   \mathrm{Iso}(G_1,G_2)$ does not coincide with $\mathrm{HypIso}(G_1,G_2)$, in particular their game algebras are not canonically isomorphic. It would be still interesting to know however whether the two games are $*$-equivalent. In particular, we do not know if the graphs $ G_1, G_2$ in Proposition \ref{p:strict_incl} are $ \rm q$-isomorphic as hypergraphs via some other perfect strategy.  A similar question holds for $\mathrm{HypIso}(\cl N(G_1), \cl N(G_2))$ and $ \mathrm{Iso}(G_1,G_2)$.

 \section{Bi-Galois extensions} \label{sec_bigalois}

As it becomes apparent by Theorem \ref{prop:quantum_isomorphisms}, the  $*$-algebra $\cl A({\rm HypIso}(\Lambda_1, \Lambda_2))$ encodes quantum isomorphisms between the hypergraphs $\Lambda_{1}$ and $\Lambda_{2}$. We may thus also refer to it as the \emph{quantum isomorphism space}. 
We note that by Proposition \ref{prop:game-alg-presentation}, when $\Lambda_1=\Lambda_2:= \Lambda$, then $\cl A({\rm HypIso}(\Lambda, \Lambda))= \cl O(\Lambda)$, where $\cl O(\Lambda)$ denotes the Hopf $*$-algebra associated to the quantum hypergraph automorphism group.  For this reason we denote 

\[
\cl O(\Lambda_1,\Lambda_2):= \cl A({\rm HypIso}(\Lambda_1, \Lambda_2)).
\]

Our goal in this section is to establish the implication 
\(A^{*}\Rightarrow \mathrm{qc}\) in the chain of Figure \ref{fig:iso-implications}. The proof follows ideas from \cite{bigalois}.
It suffices to show that, whenever 
\(\mathcal{O}(\Lambda_{1},\Lambda_{2})\neq 0\), the algebra 
\(\mathcal{O}(\Lambda_{1},\Lambda_{2})\) admits a left (resp. right)-invariant state. 
Indeed, in this case \(\mathcal{O}(\Lambda_{1},\Lambda_{2})\) is an 
\(\mathcal{O}(\Lambda_{1})\)–\(\mathcal{O}(\Lambda_{2})\) bi-Galois extension, 
and it will follow from \cite{bichon99,bichon-rijdt-vaes} 
 (see also Theorem~3.15 in \cite{bigalois}) that it admits a non-zero $*$-representation  as bounded operators on a Hilbert space.

\par We begin by showing that the Hopf $*$-algebra $\cl{O}(\Lambda_1)$ associated to the quantum automorphism group of $\Lambda_1$ has a left *-comodule algebra structure.

\begin{proposition}\label{p_delta_comod}
Let $\Lambda_i = (V_i, E_i)$, $i=1,2$, be hypergraphs, and $ \cl O(\Lambda_1,\Lambda_2) \neq \{0\}$. Then the map
\[
\delta := \delta_V \oplus \delta_E : \mathcal{O}(\Lambda_1, \Lambda_2) \to \mathcal{O}(\Lambda_1) \otimes \mathcal{O}(\Lambda_1, \Lambda_2)
\]
where
\[
\delta_V: \mathcal{O}(\Lambda_1, \Lambda_2) \to \mathcal{O}(\Lambda_1) \otimes \mathcal{O}(\Lambda_1, \Lambda_2), \quad
\delta_V(p_{v_1,v_2}) := \sum_{w_1 \in V_1} u_{v_1,w_1} \otimes p_{w_1,v_2},
\]
and
\[
\delta_E: \mathcal{O}(\Lambda_1, \Lambda_2) \to \mathcal{O}(\Lambda_1) \otimes \mathcal{O}(\Lambda_1, \Lambda_2), \quad
\delta_E(p_{e_1,e_2}) := \sum_{g_1 \in E_1} u_{e_1,g_1} \otimes p_{g_1,e_2}.
\]
defines a left $*$-comodule algebra structure over $\mathcal{O}(\Lambda_1)$.
\end{proposition}

\begin{proof}
It is straightforward that
\[
\delta_V(P_{V}) := (\delta_V(p_{v_1,v_2}))_{v_1 \in V_1, v_2 \in V_2} \in M_{V_1, V_2}(\mathcal{O}(\Lambda_1) \otimes \mathcal{O}(\Lambda_1, \Lambda_2)),
\]
\[
\delta_E(P_{E}) := (\delta_E(p_{e_1,e_2}))_{e_1 \in E_1, e_2 \in E_2} \in M_{E_1, E_2}(\mathcal{O}(\Lambda_1) \otimes \mathcal{O}(\Lambda_1, \Lambda_2)),
\]
are quantum permutation matrices.
Next, using the following intertwining relations for $u_V$, $u_E$ and $P_{V}$, $P_{E}$
\[
A_{\Lambda_1} u_E = u_V A_{\Lambda_1}, \quad A_{\Lambda_1} P_{E} = P_{V} A_{\Lambda_2},
\]
we show that 
$$(A_{\Lambda_1}\otimes1)\delta_E(P_{E})=\delta_V(P_{V})(1\otimes A_{\Lambda_2})$$
by equivalently showing that 
\[
\sum_{e_1\in E_1, v_1\in e_1} \sum_{g_1\in E_1} u_{e_1,g_1} \otimes p_{g_1,f_2}  = \sum_{v_2\in V_2, v_2 \in f_2} \sum_{w_1\in V_1} u_{v_1,w_1} \otimes p_{w_1,v_2},
\]
due to Lemma \ref{l_inclusion_rel}.
We compute
\begin{align*}
\sum_{e_1\in E_1, v_1\in e_1} \sum_{g_1\in E_1} u_{e_1,g_1} \otimes p_{g_1,f_2}
&=  \sum_{g_1\in E_1} \bigg(\sum_{e_1\in E_1, v_1\in e_1}u_{e_1,g_1}\bigg) \otimes p_{g_1,f_2}\\ &
= \sum_{g_1\in E_1} \bigg(\sum_{w_1\in V_1, w_1\in g_1}u_{v_1,w_1}\bigg) \otimes p_{g_1,f_2}\\ & =\sum_{w_1\in V_1} u_{v_1,w_1} \otimes \bigg(\sum_{g_1\in E_1,w_1\in g_1} p_{g_1,f_2}\bigg)\\
&=\sum_{w_1\in V_1} u_{v_1,w_1} \otimes \bigg(\sum_{v_2 \in V_2,v_2\in f_2} p_{w_1,v_2}\bigg)\\
&=\sum_{v_2\in V_2, v_2 \in f_2} \sum_{w_1\in V_1} u_{v_1,w_1} \otimes p_{w_1,v_2}.
\end{align*}
Thus, the incidence matrix intertwining is preserved by the coactions. By the universality of the unital $*$-algebra $\cl O(\Lambda_1,\Lambda_2)$ we conclude that $\delta$ defines a unital *-homomorphism.
\par For coassociativity, note that for any vertex generator $p_{v_1,v_2}$, 
\begin{align*}
((\Delta \otimes \mathrm{id}) \circ \delta)(p_{v_1,v_2}) 
&= (\Delta \otimes \mathrm{id}) \left( \sum_{w_1 \in V_1} u_{v_1,w_1} \otimes p_{w_1,v_2} \right) \\
&= \sum_{w_1 \in V_1} \Delta(u_{v_1,w_1}) \otimes p_{w_1,v_2} \\
&= \sum_{w_1 \in V_1} \left( \sum_{x_1 \in V_1} u_{v_1,x_1} \otimes u_{x_1,w_1} \right) \otimes p_{w_1,v_2} \\
&= \sum_{x_1,w_1 \in V_1} u_{v_1,x_1} \otimes u_{x_1,w_1} \otimes p_{w_1,v_2}.
\end{align*}
On the other hand,
\begin{align*}
((\mathrm{id} \otimes \delta) \circ \delta)(p_{v_1,v_2}) 
&= (\mathrm{id} \otimes \delta) \left( \sum_{w_1 \in V_1} u_{v_1,w_1} \otimes p_{w_1,v_2} \right) \\
&= \sum_{w_1 \in V_1} u_{v_1,w_1} \otimes \delta(p_{w_1,v_2}) \\
&= \sum_{w_1 \in V_1} u_{v_1,w_1} \otimes \left( \sum_{y_1 \in V_1} u_{w_1,y_1} \otimes p_{y_1,v_2} \right)\\
&= \sum_{w_1,y_1 \in V_1} u_{v_1,w_1} \otimes u_{w_1,y_1} \otimes p_{y_1,v_2}.
\end{align*}
Renaming indices $(x_1,w_1) \leftrightarrow (w_1,y_1)$ shows that the two expressions are equal. The argument for edge generators $p_{e_1,e_2}$ is identical, using the coproduct properties of $u_E = (u_{e_1,g_1})$. For counitality note that on vertex generators,
\[
\begin{aligned}
((\varepsilon \otimes \mathrm{id}) \circ \delta)(p_{v_1,v_2}) 
&= (\varepsilon \otimes \mathrm{id}) \left( \sum_{w_1 \in V_1} u_{v_1,w_1} \otimes p_{w_1,v_2} \right) \\
&= \sum_{w_1 \in V_1} \varepsilon(u_{v_1,w_1}) p_{w_1,v_2}= \sum_{w_1 \in V_1} \delta_{v_1,w_1} p_{w_1,v_2} = p_{v_1,v_2},
\end{aligned}
\]
for any $v_1 \in V_{1}, v_2 \in V_{2}$. The same holds for the edge generators \(p_{e_1,e_2}\), $e_1 \in E_{1}, e_2 \in E_{2}$. 

Since \(\delta\) is a *-homomorphism extending from these generators, both co-associativity and co-unitality hold on the entire algebra \(\mathcal{O}(\Lambda_1, \Lambda_2)\).
\end{proof}

\begin{remark}\label{rm_gamma_comod} \rm
Using analogous arguments, one can define a right coaction
\[
\gamma: \mathcal{O}(\Lambda_1, \Lambda_2) \to \mathcal{O}(\Lambda_1, \Lambda_2) \otimes \mathcal{O}(\Lambda_2)
\]
on the generators by
\[
\gamma(p_{v_1,v_2}) := \sum_{w_2 \in V_2} p_{v_1,w_2} \otimes u_{w_2,v_2}', \quad
\gamma(p_{e_1,e_2}) := \sum_{f_2 \in E_2} p_{e_1,f_2} \otimes u_{f_2,e_2}',
\]
where \(u' = (u_{w_2,v_2}')\) and \(u'_E = (u_{f_2,e_2}')\) denote the quantum permutation matrices generating \(\mathcal{O}(\Lambda_2)\), and show that it defines a right \(*\)-comodule algebra structure of \(\mathcal{O}(\Lambda_1, \Lambda_2)\) over \(\mathcal{O}(\Lambda_2)\), compatible with the incidence matrix relations.
\end{remark}

\begin{theorem} \label{th_bigalois_ext}
If  \(\mathcal{O}(\Lambda_1, \Lambda_2) \neq 0\),  then \(\mathcal{O}(\Lambda_1, \Lambda_2)\) is a \(\mathcal{O}(\Lambda_1)\)--\(\mathcal{O}(\Lambda_2)\) bi-Galois extension via the left coaction 
    \[
    \delta(p_{v_1,v_2}) = \sum_{w_1 \in V_1} u_{v_1,w_1} \otimes p_{w_1,v_2}, \quad \delta(p_{e_1,e_2}) = \sum_{d_1 \in E_1} u_{e_1,d_1} \otimes p_{d_1,e_2},
    \]
and the right coaction
    \[
    \gamma(p_{v_1,v_2}) = \sum_{w_2 \in V_2} p_{v_1,w_2} \otimes v_{w_2,v_2}, \quad \gamma(p_{e_1,e_2}) = \sum_{f_2 \in E_2} p_{e_1,f_2} \otimes v_{f_2,e_2}.
    \]
\end{theorem}

\begin{proof}
First note that by Proposition \ref{p_delta_comod} and Remark \ref{rm_gamma_comod}, $\delta$ (resp. $\gamma$) defines a left (resp. right) *-comodule algebra structure over $\cl{O}(\Lambda_1)$ (resp. $\cl{O}(\Lambda_2))$. To conclude the proof we must check that the coactions commute, and canonical maps $\kappa_{\delta}, \kappa_{\gamma}$  are bijections.

To show the coactions commute: direct computation yields
\[
(\delta \otimes \mathrm{id}) \circ \gamma(p_{v_1,v_2}) = \sum_{w_2, w_1} u_{v_1,w_1} \otimes p_{w_1,w_2} \otimes v_{w_2,v_2},
\]
\[
(\mathrm{id} \otimes \gamma) \circ \delta(p_{v_1,v_2}) = \sum_{w_1,w_2} u_{v_1,w_1} \otimes p_{w_1,w_2} \otimes v_{w_2,v_2}.
\]
Indeed, we verify the identity on the generator \(p_{v_1,v_2}\). First, compute the left-hand side:
\[
\begin{aligned}
(\delta \otimes \mathrm{id}) \circ \gamma(p_{v_1,v_2}) 
&= (\delta \otimes \mathrm{id})\left( \sum_{w_2 \in V_2} p_{v_1,w_2} \otimes v_{w_2,v_2} \right) = \sum_{w_2 \in V_2} \delta(p_{v_1,w_2}) \otimes v_{w_2,v_2} \\
&= \sum_{w_2 \in V_2} \sum_{w_1 \in V_1} u_{v_1,w_1} \otimes p_{w_1,w_2} \otimes v_{w_2,v_2} = \sum_{w_1 \in V_1} \sum_{w_2 \in V_2} u_{v_1,w_1} \otimes p_{w_1,w_2} \otimes v_{w_2,v_2}.
\end{aligned}
\]
Now compute the right-hand side:
\[
\begin{aligned}
(\mathrm{id} \otimes \gamma) \circ \delta(p_{v_1,v_2}) 
&= (\mathrm{id} \otimes \gamma)\left( \sum_{w_1 \in V_1} u_{v_1,w_1} \otimes p_{w_1,v_2} \right) = \sum_{w_1 \in V_1} u_{v_1,w_1} \otimes \gamma(p_{w_1,v_2}) \\
&= \sum_{w_1 \in V_1} u_{v_1,w_1} \otimes \sum_{w_2 \in V_2} p_{w_1,w_2} \otimes v_{w_2,v_2} = \sum_{w_1 \in V_1} \sum_{w_2 \in V_2} u_{v_1,w_1} \otimes p_{w_1,w_2} \otimes v_{w_2,v_2}.
\end{aligned}
\]
Since both expressions agree, the coactions commute on the vertex generators. An analogous computation applies to the edge generators \(p_{e_1,e_2}\), establishing our claim.
\

To show the canonical maps \(\kappa_{\delta}\) and \(\kappa_{\gamma}\) are bijective, we define the maps
\[
\eta_\delta := (\mathrm{id} \otimes m) \circ (\alpha \otimes \mathrm{id}): \mathcal{O}(\Lambda_1) \otimes \mathcal{O}(\Lambda_1, \Lambda_2) \to \mathcal{O}(\Lambda_1, \Lambda_2) \otimes \mathcal{O}(\Lambda_1, \Lambda_2),
\]
where \(m: \mathcal{O}(\Lambda_1, \Lambda_2) \otimes \mathcal{O}(\Lambda_1, \Lambda_2) \to \mathcal{O}(\Lambda_1, \Lambda_2)\) is the multiplication map
and 
\[
\alpha : \mathcal{O}(\Lambda_1) \to \mathcal{O}(\Lambda_1, \Lambda_2) \otimes \mathcal{O}(\Lambda_1, \Lambda_2)
\]
is the unital $*$-homomorphism
\[
\alpha(u_{v_1, w_1})=\sum_{v_2\in V_2}p_{v_1, v_2}\otimes p_{w_1, v_2}^{}.
\]
Similarly, define
\[
\eta_{\gamma} := (m \otimes \mathrm{id}) \circ (\mathrm{id} \otimes \beta): \mathcal{O}(\Lambda_1, \Lambda_2) \otimes \mathcal{O}(\Lambda_2) \to \mathcal{O}(\Lambda_1, \Lambda_2) \otimes \mathcal{O}(\Lambda_1, \Lambda_2),
\]
\noindent where $\beta: \cl{O}(\Lambda_{2})\rightarrow \cl{O}(\Lambda_{1}, \Lambda_{2})\otimes \cl{O}(\Lambda_{1}, \Lambda_{2})$ is the map analogous to $\alpha$, but defined on the generators of $\cl{O}(\Lambda_{2})$. 

We will show that $\kappa_{\delta}$ and $\eta_{\delta}$ (resp. $k_\gamma$ and $\eta_{\gamma}$) are mutual inverses. We verify the equalities on the generators indexed by the vertices; the case of edges follows similarly. For the left case, consider $u_{v_1,w_1} \in \mathcal{O}(\Lambda_1)$ and $p_{x_1,x_2} \in \mathcal{O}(\Lambda_1, \Lambda_2)$. Then:
\begin{align*}
    \eta_{\delta}(u_{v_1,w_1} \otimes p_{x_1,x_2}) &= (\mathrm{id} \otimes m) \circ (\alpha \otimes \mathrm{id})(u_{v_1,w_1} \otimes p_{x_1,x_2})\\
    &= \sum_{y_2\in V_2}p_{v_1,y_2} \otimes p_{w_1,y_2} p_{x_1,x_2}.
\end{align*}
Then:
\begin{align*}
    \kappa_{\delta}(\eta_{\delta}(u_{v_1,w_1} \otimes p_{x_1,x_2})) &= \sum_{y_2 \in V_2}\delta(p_{v_1,y_2})(1 \otimes p_{w_1,y_2}p_{x_1,x_2}) \\
    &= \sum_{y_1 \in V_1}\sum_{y_2 \in V_2}u_{v_1,y_1}\otimes p_{y_1,y_2}p_{w_1,y_2}p_{x_1,x_2} \\
    &=\sum_{y_1 \in V_1}\sum_{y_2 \in V_2}u_{v_1,y_1}\otimes \delta_{y_1,w_1} p_{y_1,y_2}p_{x_1,x_2}\\
    &= \sum_{y_2 \in V_2}u_{v_1,w_1}\otimes  p_{w_1,y_2}p_{x_1,x_2}= u_{v_1,w_1}\otimes p_{x_1,x_2}
\end{align*}
so $\kappa_{\delta} \circ \eta_{\delta} = \mathrm{id}$.

Now take $p_{v_1,v_2} \otimes p_{w_1,w_2}$:
\[
\kappa_{\delta}(p_{v_1,v_2} \otimes p_{w_1,w_2}) = \sum_{x_1 \in V_1} u_{v_1,x_1} \otimes p_{x_1,v_2}p_{w_1,w_2},
\]
Applying $\eta_{\delta}$ gives:
\begin{align*}
    \sum_{x_1 \in V_1} \eta_{\delta}(u_{v_1,x_1} \otimes p_{x_1,v_2}p_{w_1,w_2})&= \sum_{x_1 \in V_1} (\id \otimes m)(\alpha \otimes \id)(u_{v_1,x_1} \otimes p_{x_1,v_2}p_{w_1,w_2})\\
    &=\sum_{x_1 \in V_1} \sum_{y_2\in V_2} p_{v_1,y_2}\otimes p_{x_1,y_2}p_{x_1,v_2}p_{w_1,w_2} \\
    &= p_{v_1,v_2}\otimes p_{w_1,w_2}.
\end{align*}
Thus, $\eta_{\delta} \circ \kappa_{\delta} = \mathrm{id}$.

The right case is analogous. 
\end{proof}

Now, we show the existence of C*-representations of the quantum isomorphism space.

\begin{theorem}\label{t_nonzero_implies_trace}
Let \( \Lambda_1 \) and \( \Lambda_2 \) be two  hypergraphs, and suppose that the quantum isomorphism space is non-zero, i.e., \( \mathcal{O}(\Lambda_1, \Lambda_2) \neq 0 \). Then \( \mathcal{O}(\Lambda_1, \Lambda_2) \) admits a non-zero $*$-representation  as bounded operators on a Hilbert space. In particular, it admits a faithful bi-invariant tracial state.
\end{theorem}
\begin{proof}
Let $\Lambda_1 = (V_1, E_1)$ and $\Lambda_2 = (V_2, E_2)$ be hypergraphs. Define $n = |V_1| + |E_1|$ and $m = |V_2| + |E_2|$. Consider the block matrix
\( 
u =u_{V_1} \oplus u_{E_1} \in M_n(\mathcal{O}(\Lambda_1)),
\)
where $u_{V_1}=(u_{v_1,w_1})$ and $u_{E_1}=(u_{e_1,f_1})$ are the generating quantum permutation matrices for vertices and edges. Since each block of $u$ is unitary, so is $u$, and by definition of $A_u^0(n)$, this yields a surjective Hopf $*$-algebra morphism
\[
\pi: A_u^0(n) \twoheadrightarrow \mathcal{O}(\Lambda_1),
\]
making $\cl O(\Lambda_1)$ a quantum subgroup of $A_u^0(n)$. We recall that $ A^0_{u}(n)$ (resp. $ A_{u}^0(n,m)$) is the  unital $*$-algebra generated by the entries of  an $n$ by $n$ (resp. $n $ by $ m$) matrix $u$ such that both $ u$ and $ \bar u$ are unitaries.

Similarly, define the block matrix
\(  P_{V} \oplus P_{E}
 \in M_{n \times m}(\mathcal{O}(\Lambda_1, \Lambda_2)),
\)
where quantum permutation matrices $P_{V} = (p_{v_1,v_2})_{v_1 \in V_1, v_2 \in V_2}$ and $P_{E} = (p_{e_1,e_2})_{e_1 \in E_1, e_2 \in E_2}$ satisfy the incidence compatibility relation
\[
A_{\Lambda_1} P_{E} = P_{V} A_{\Lambda_2}.
\]
Hence, there exists a surjective $*$-homomorphism
\[
\sigma: A_u^0(n,m) \twoheadrightarrow \mathcal{O}(\Lambda_1, \Lambda_2),
\]
sending the generators of $A_u^0(n,m)$ to the corresponding entries of $P_{V}\oplus P_E$. Denote by $\alpha_{n,m}$ the canonical left coaction
\[
\alpha_{n,m}: A_u^0(n,m) \to A_u^0(n) \otimes A_u^0(n,m)
\]
such that
\[
\alpha_{n,m}(u_{ij}) = \sum_{k=1}^{n} u^{(n)}_{ik} \otimes u_{kj}.
\]
We verify that
\begin{center}
\begin{tikzcd}[column sep=huge, row sep=large]
A_u^0(n,m) \arrow[r, two heads, "\sigma"] \arrow[d, "\alpha_{n,m}"'] & \mathcal{O}(\Lambda_1, \Lambda_2) \arrow[d, "\delta"] \\
A_u^0(n) \otimes A_u^0(n,m) \arrow[r, two heads, "\pi \otimes \sigma"'] & \mathcal{O}(\Lambda_1) \otimes \mathcal{O}(\Lambda_1, \Lambda_2)
\end{tikzcd}
\end{center}
where $\delta$ is the left coaction of $\mathcal{O}(\Lambda_1)$ on $\mathcal{O}(\Lambda_1, \Lambda_2)$, given on generators by
\[
\delta(p_{v_1v_2}) = \sum_{w_1 \in V_1} u_{v_1w_1} \otimes p_{w_1v_2}, \quad
\delta(p_{e_1e_2}) = \sum_{g_1 \in E_1} u_{e_1g_1} \otimes p_{g_1e_2}.
\]
Evaluating both sides on a generator $u_{ij} \in A_u^0(n,m)$, we compute:
\[
(\delta \circ \sigma)(u_{ij}) = \delta(p_{ij}) = \sum_{k=1}^{n} \pi(u^{(n)}_{ik}) \otimes p_{kj},
\]
\[
((\pi \otimes \sigma) \circ \alpha_{n,m})(u_{ij}) = \sum_{k=1}^{n} \pi(u^{(n)}_{ik}) \otimes \sigma(u_{kj}) = \sum_{k=1}^{n} \pi(u^{(n)}_{ik}) \otimes p_{kj}.
\]
Since, the identity holds on generators, and we have the desired result.
Hence, by \cite[Proposition 6.2.6]{bichon99} $ \cl O(\Lambda_1,\Lambda_2)$ admits a left-invariant state and by \cite[Theorem 5.2.1]{bichon99} it admits a unital $*$-representation into $\cl B(H)$ for some Hilbert space $H$. 
In particular we note that by \cite[Theorem 3.15]{bigalois} this is equivalent to the existence of a faithful bi-invariant state $\tau: \cl O(\Lambda_1,\Lambda_2) \to \bb C$. We note that $O(\Lambda_1)$ and $ \cl O(\Lambda_2)$ are of Kac type by virtue of \cite[Prop.\ 1.7.9]{Neshtuset} since the antipode $ S$ on the associated Hopf $*$-algebra is $*$-preserving. Thus, $\tau$ is moreover a trace.
\end{proof}
We thank the anonymous reviewer for bringing to our attention the following alternative approaches. 
\begin{remark}\label{rm:alternative_proof_reviewer}
\rm By viewing the hypergraphs $\Lambda_{i}, i = 1, 2$ as directed bipartite graphs with loops one can directly apply the results of \cite{bigalois} to obtain an alternative proof for Theorem~\ref{t_nonzero_implies_trace}. 
\par First, assume that $ \mathcal{A}({\rm HypIso}(\Lambda_{1}, \Lambda_{2})) \neq \{0\}$, and consider $\Lambda_1$ (resp. $\Lambda_2)$ as a bipartite graph where the vertices are $V_i\sqcup E_i$, $i=1,2$ and connect the vertices to the hyperedges to which they belong. Moreover, add a loop to each $u\in V_i$ to be able to distinguish them from the vertices coming from the hyperedges. Note that the adjacency matrices of the newly constructed graphs are given by
\begin{gather*}
   A_{i} := \begin{bmatrix} I_{V_{i}} & A_{\Lambda_{i}} \\ 0 & 0\end{bmatrix}, \;\;\;\; i = 1, 2,
\end{gather*}
\noindent where the block-zeros in the above matrices are chosen such that $A_{i}$ is a $(\lvert V_{i}\rvert +\lvert E_i\rvert)\times (\lvert V_{i}\rvert +\lvert E_i\rvert)$ matrix $i = 1, 2$. If $P$ denotes the (universal) magic unitary for the graph isomorphism game ${\rm Iso}(A_{1}, A_{2})$ generating $\mathcal{A}({\rm Iso}(A_{1}, A_{2}))$, we decompose it in block form as 
\begin{gather*}
    P = \begin{bmatrix} P_{V} & P_{VE} \\ P_{EV} & P_{E}\end{bmatrix},
\end{gather*}
\noindent which yields (under the intertwining relations)
\begin{gather*} 
    \begin{bmatrix}
        P_{V}+A_{\Lambda_{1}}P_{EV} & P_{VE}+A_{\Lambda_{1}}P_{E} \\ 0 & 0
    \end{bmatrix} = A_{1}P = PA_{2} = \begin{bmatrix} P_{V} & P_{V}A_{\Lambda_{2}} \\ P_{EV} & P_{EV}A_{\Lambda_{2}}\end{bmatrix}.
\end{gather*}
\noindent Automatically, we see $P_{EV} = 0$. As $P$ preserves the unit (in that $P1_{V_{2}\sqcup E_{2}} = 1_{V_{1}\sqcup E_{1}}$), we compute
\begin{eqnarray*}
    \langle P_{VE}1_{E_{2}}, 1_{V_{1}}\rangle 
    & = &
    \langle [P_{V} \;P_{VE}]1_{V_{2}\sqcup E_{2}}, 1_{V_{1}}\rangle-\langle P_{V}1_{V_{2}}, 1_{V_{1}}\rangle \\
    & = &
    \bigg\langle [P_{V} \; P_{VE}]1_{V_{2}\sqcup E_{2}}, 1_{V_{1}}\bigg\rangle-\bigg\langle1_{V_{2}}, [  P_{V}^{*} \; P_{EV}^* ] 1_{V_{1} \sqcup E_{1}}\bigg\rangle  \\
    & = &
    \langle 1_{V_{1}}, 1_{V_{1}}\rangle-\langle 1_{V_{2}}, 1_{V_{2}}\rangle \\
    & = &
    |V_{1}|-|V_{2}|,
\end{eqnarray*}
\noindent which by Remark \ref{rmk_samesizes} equals to 0, and thus $\sum\limits_{v \in V_{1}, e \in E_{2}}p_{v, e} = 0$. This means $P_{VE}$ vanishes in every $C^*$-representation and, in particular, $\mathcal{A}({\rm Iso}(A_{1}, A_{2}))$ has the same universal $C^*$-algebra. Therefore,  
\begin{gather*}
    \mathcal{A}({\rm HypIso}(\Lambda_{1}, \Lambda_{2}))  \cong \mathcal{A}({\rm Iso}(A_{1}, A_{2}))/ \langle P_{VE} =0 \rangle.
\end{gather*}
Thus, as $ \mathcal{A}({\rm HypIso}(\Lambda_{1}, \Lambda_{2})) \neq \{0\}$, we have that $\mathcal{A}({\rm Iso}(A_{1}, A_{2})) \neq \{0\}$, which by \cite[Theorem 3.15]{bigalois} implies $\mathcal{A}({\rm Iso}(A_{1}, A_{2}))$ (and thus $\mathcal{O}(\Lambda_{1}, \Lambda_{2})=  \mathcal{A}({\rm HypIso}(\Lambda_{1}, \Lambda_{2}))$) possesses a faithful, bi-invariant tracial state.

\smallskip

\par From another viewpoint, as pointed out by the referee, the hypergraph isomorphism game may be regarded as a coloured graph isomorphism game. Indeed, on the vertex set $V_i\sqcup E_i$, assign one colour to the directed incidence edges $v\to e$ whenever $v\in e$, a second colour to the loops at the vertices in $V_i$, and a third colour to the loops at the vertices in $E_i$. The adjacency matrices corresponding to these three colours are the block embeddings of $A_{\Lambda_i}$, $I_{V_i}$, and $I_{E_i}$, respectively. Equivalently, this is the incidence graph with two vertex colours distinguishing $V_i$ from $E_i$.  Adapting the decolouring techniques of \cite{drs25} (see also \cite{bgmsw}), one obtains ordinary uncoloured graphs $G_1$ and $G_2$ such that
\( 
\mathcal{A}({\rm HypIso}(\Lambda_1,\Lambda_2))
\cong
\mathcal{A}({\rm Iso}(G_1,G_2)).
\)
This provides another perspective on the preceding reduction and suggests a natural extension to vertex-multicoloured isomorphism games.
\end{remark}

\begin{theorem}\label{t_qc_strats_iso_game}
    Let $\Lambda_{i} = (V_{i}, E_{i}), i = 1, 2$ be hypergraphs. The following are equivalent:
    \begin{enumerate} 
        \item[(i)] $\Lambda_1 \cong_{\rm qc} \Lambda_2$,
        \item[(ii)] $ \Lambda_1 \cong_{ {\rm C}^*} \Lambda_2$,
        \item[(iii)] $\Lambda_1 \cong_{ A^*} \Lambda_2 $.
    \end{enumerate}
\end{theorem}
\begin{proof}
 (i) $\Rightarrow$ (ii) is clear.
\noindent

(ii) $\Rightarrow$ (iii) is immediate.

   (iii) $\Rightarrow$ (i) Since $\Lambda_1 \cong_{ A^*} \Lambda_2 $ means by definition that $\cl{O}(\Lambda_{1}, \Lambda_{2}) \neq \{0\}$, then by Theorem \ref{t_nonzero_implies_trace} there exists a faithful bi-invariant tracial state $\tau: \cl{O}(\Lambda_{1}, \Lambda_{2})\rightarrow \bb{C}$. By \cite[Theorem 5.2.1]{bichon99}, the GNS representation of $\pi_{\tau}: \cl{O}(\Lambda_{1}, \Lambda_{2}) \rightarrow \cl{B}(H_{\tau})$ is well defined. Considering $\cl{A} := \overline{\pi_{\tau}(\cl{O}(\Lambda_{1}, \Lambda_{2}))}^{\|\cdot\|}$ (i.e., the closure of the image of the $*$-algebra $\cl{O}(\Lambda_{1}, \Lambda_{2})$ in $\cl{B}(H_{\tau})$), we have the existence of a ${\rm C}^{*}$-algebra endowed with a faithful tracial state. Furthermore, if we set $P_{V}$ and $P_{E}$ to be the images of the generating magic unitaries for $\cl{O}(\Lambda_{1}, \Lambda_{2})$ under $\pi_{\tau}$, by the definition of the quantum isomorphism space it is immediate that $P_{V}, P_{E}$ are quantum permutation matrices with entries in the ${\rm C}^{*}$-algebra satisfying the desired conditions.
\end{proof}

\begin{example}\label{ex_n_k}
\rm  For $n, k \in \bb{N}$ let $\Lambda_{n, k} := ([n], E_{k})$ be the hypergraph on $n$ vertices, where $E_{k}$ is $k$ distinct copies of $[n]$. We claim that for any hypergraph $\Lambda = (V, E)$, then $\cl{O}(\Lambda, \Lambda_{n, k}) \neq \{0\}$ if and only if $\Lambda \cong \Lambda_{n, k}$. One direction is trivial; for the other, assume towards contradiction that $\cl{O}(\Lambda, \Lambda_{n, k}) \neq \{0\}$ while $\Lambda \not \cong \Lambda_{n, k}$. Let $P_{V}, P_{E}$ be the magic unitaries generating $\cl{O}(\Lambda, \Lambda_{n, k})$, and denote by $J_{n, k}$ the $n\times k$ matrix with all $1$ entries. Note that $ A_{\Lambda_{n,k}}= J_{n,k}$ and by the intertwining relations, we have
\begin{gather*}
    A_{\Lambda}P_{E} = P_{V}J_{n, k} = J_{V, k}.
\end{gather*}
\noindent This means for any $v \in V, e \in E_{k}$ we have
\begin{gather}\label{eqn_summing_rel}
    \sum\limits_{e' \in E}(A_{\Lambda})_{v, e'}\cdot p_{e'e} = 1.
\end{gather}
\noindent For now, fix $v \in V, e \in E_{k}$; as $\Lambda \not \cong \Lambda_{n, k}$, there exists at least one $e' \in E$ such that $v \not \in e'$ (inside $\Lambda$). Using (\ref{eqn_summing_rel}) and orthogonality conditions on $P_E$, this implies $p_{e'e} = 0$ for $e \in E_{k}$. As our choice of $e \in E_{k}$ was arbitrary, this holds for every hyperedge in $E_{k}$; thus, for this specific $e' \in E$, $p_{e'e} = 0$ for all $e \in E_{k}$. However, as $P_{E}$ is a magic unitary, we have
\begin{gather*}
    1 = \sum\limits_{e \in E_{k}}p_{e'e} = \sum\limits_{e \in E_{k}}0 = 0,
\end{gather*}
\noindent a clear contradiction. Therefore, $\cl{O}(\Lambda, \Lambda_{n, k}) = \{0\}$. By Theorem \ref{t_qc_strats_iso_game}, this implies for ${\rm t} \in \{\rm loc, q, qa, qc\}$ we have $\Lambda \cong_{\rm t} \Lambda_{n, k}$ if and only if $\Lambda \cong \Lambda_{n, k}$. 
\end{example}


\section{The game isomorphism game and transfer of strategies}\label{s_nlg_trans}
In this section we switch notation to match the approach taken in \cite{ht_one}. A \emph{hypergraph} is a subset \(\Lambda\subseteq X\times A\) (finite sets \(X,A\)). Vertices are represented by the set \(X\), while edges are the subsets \(\{\Lambda (a):a\in A\}\) with \(\Lambda(a):=\{x\in X:(x,a)\in \Lambda \}\). The incidence matrix $A_{\Lambda}$ is defined as before, as an $X\times A$ matrix with entries in $\{0, 1\}$ where $(A_{\Lambda})_{x, a} = 1$ if and only if $(x, a) \in \Lambda$.  
\begin{remark}
\rm To connect the two presentations of hypergraphs, if we start with a subset $\Lambda \subseteq V\times E$ of finite sets $V$ and $E$, for each $e \in E$ define
\begin{gather*}
    E_{e} := \{v \in V: \; (v, e) \in \Lambda\}.
\end{gather*}
\noindent Then $\tilde{\Lambda} = (V, \{E_{e}\}_{e \in E})$ is a presentation of $\Lambda$ as in the previous sections. Similarly, for $\Lambda = (V, E)$ let
\begin{gather*}
    \tilde{\Lambda} := \{(v, e): \; v \in V, e \in E, v \in e\} \subseteq V\times E.
\end{gather*}
This shows that either presentation contains the same incidence relations for a given hypergraph $\Lambda$. 
\end{remark}

Let 
\(
\cl G_{i}=(X_i,Y_i;\,A_i,B_i,\lambda_i)\), \(   i=1,2,\)
be non-local games. Encode each game as a hypergraph of winning pairs:
\[
\Lambda_{\cl G_{i}}\ \subseteq\ (X_i\times Y_i)\times(A_i\times B_i),\qquad
\bigl((x_i,y_i),(a_i,b_i)\bigr)\in\Lambda_{\cl G_{i}}\iff \lambda_i(x_i,y_i,a_i,b_i)=1.
\]
We now consider the hypergraph isomorphism game as a ``game of games". Note that the questions and answers sets for the game of games are the sets \( (X_1 \times Y_1) \sqcup (X_2 \times Y_2) \sqcup  (A_1\times B_1) \sqcup (A_2 \times B_2)\).

\subsection{The NS game $*$-algebra}

Fix two non-local games \(\cl G_{i}=(X_i,Y_i;A_i,B_i,\lambda_i)\) with
\(|X_1 \times Y_1|=|X_2 \times Y_2|\),  \(|A_1\times B_1|=|A_2 \times B_2|\).
Write \(U_i:=X_i\times Y_i\) and \(W_i:=A_i\times B_i\).
Let \(\mathcal O(\Lambda_{\cl G_1}, \Lambda_{\cl G_2})\) be the quantum isomorphism space, that is, the  unital \(^*\)-algebra
generated by two magic unitaries
\[
P_U=\bigl(p^U_{(x_1,y_1),\,(x_2,y_2)}\bigr)_{U_1\times U_2},\qquad
P_W=\bigl(p^W_{(a_1,b_1),\,(a_2,b_2)}\bigr)_{W_1\times W_2},
\]
subject to the incidence intertwining relation
\begin{equation*}
A_{\Lambda_{1}}\,P_{W}\;=\;P_{U}\,A_{\Lambda_2},
\end{equation*}
where $ \Lambda_i= \Lambda_{\cl G_i}$. 

\begin{definition}\label{d:NS-quotient} \rm 
Let \(I_{\mathrm{NS}}\) be the \(^*\)-ideal in $ \mathcal O(\Lambda_{\cl G_1}, \Lambda_{\cl G_2})$ generated by the
operator no–signalling marginal equalities:
for the $U$-block (questions), require
\begin{align*}
\sum_{x_1\in X_1} p^U_{(x_1,y_1),\,(x_2,y_2)}
&=\sum_{x_1\in X_1} p^U_{(x_1,y_1),\,(x_2',y_2)}
&&(\forall\,y_1,y_2,\ x_2,x_2'),\\
\sum_{y_1\in Y_1} p^U_{(x_1,y_1),\,(x_2,y_2)}
&=\sum_{y_1\in Y_1} p^U_{(x_1,y_1),\,(x_2,y_2')} 
&&(\forall\,x_1,x_2,\ y_2,y_2'); 
\end{align*}
for the $W$-block (answers), require
\begin{align*}
\sum_{a_2\in A_2} p^W_{(a_1,b_1),\,(a_2,b_2)}
&=\sum_{a_2\in A_2} p^W_{(a_1',b_1),\,(a_2,b_2)}
&&(\forall\,b_1,b_2,\ a_1,a_1'), \\
\sum_{b_2\in B_2} p^W_{(a_1,b_1),\,(a_2,b_2)}
&=\sum_{b_2\in B_2} p^W_{(a_1,b_1'),\,(a_2,b_2)}
&&(\forall\,a_1,a_2,\ b_1,b_1'). 
\end{align*}
Set the NS–quotient
\[
\mathcal O_{\mathrm{NS}}(\cl G_1,\cl G_2)\ :=\ \mathcal O(\Lambda_{\cl G_1}, \Lambda_{\cl G_2})\big/ I_{\mathrm{NS}}.
\]
\end{definition}

\noindent These equalities make the following marginals well-defined:
\[
p^X_{x_1,x_2}:=\sum_{y_1} p^U_{(x_1,y_1),\,(x_2,y_2)},\quad
p^Y_{y_1,y_2}:=\sum_{x_1} p^U_{(x_1,y_1),\,(x_2,y_2)},
\]
\[
p^A_{a_1,a_2}:=\sum_{b_2} p^W_{(a_1,b_1),\,(a_2,b_2)},\quad
p^B_{b_1,b_2}:=\sum_{a_2} p^W_{(a_1,b_1),\,(a_2,b_2)},
\]
where the right-hand sides are independent of $y_2,x_2$ and $b_1,a_1$, respectively.

\begin{lemma}\label{l:NS-decomp}
In the NS-quotient algebra $\mathcal O_{\mathrm{NS}}(\cl G_1,\cl G_2)$, we have that $ (p^X_{x_1,x_2})_{x_1 \in X_1}$, $ (p^Y_{y_1,y_2})_{y_1\in Y_1}$, $ (p^A_{a_1,a_2})_{a_2\in A_2}$ and $ (p^B_{b_1,b_2})_{b_2 \in B_2}$ are PVM's for all $ x_2, y_2, a_1, b_1$ resp. and 
\[
p^U_{(x_1,y_1),\,(x_2,y_2)} = p^X_{x_1,x_2}\,p^Y_{y_1,y_2},
\qquad
p^W_{(a_1,b_1),\,(a_2,b_2)} = p^A_{a_1,a_2}\,p^B_{b_1,b_2}.
\]
If moreover, \(|X_1|=|X_2|\), \(|Y_1|=|Y_2|\), \(|A_1|=|A_2|\), \(|B_1|=|B_2|\), then
$ P_{X}=(p^X_{x_1,x_2})_{x_1,x_2}$, $ P_{Y}=(p^Y_{y_1,y_2})_{y_1,y_2}$, $ P_{A}=(p^A_{a_1,a_2})_{a_1,a_2}$ and $ P_{B}=(p^B_{b_1,b_2})_{b_1,b_2}$ are magic unitaries.
\end{lemma}

\begin{proof}
Selfadjointness, idempotency, orthogonality and partition of unity  follow immediately from the properties of the magic unitaries $P^U, P^W$.  
Row orthogonality of the magic unitary $P_U$ gives
\[
p^X_{x_1,x_2} p^Y_{y_1,y_2}
=\sum_{y_1',x_1'} p^U_{(x_1,y_1'),\,(x_2,y_2)}\,p^U_{(x_1',y_1),\,(x_2,y_2)}
= p^U_{(x_1,y_1),\,(x_2,y_2)},
\]
since all terms vanish except $(x_1',y_1')=(x_1,y_1)$; the reversed product follows by selfadjointness of the entries of $P^U$.
The $W$-block is analogous.

Now assume \(|X_1|=|X_2|\), \(|Y_1|=|Y_2|\), \(|A_1|=|A_2|\), \(|B_1|=|B_2|\). We show that the matrix $ P_{X}=(p^X_{x_1,x_2})_{x_1,x_2}$ is a magic unitary. The rest of them will follow by symmetry. Selfadjointness and idempotency is shown earlier. We check the row partition and column orthogonality. We compute
\begin{align*}
\sum_{x_2}p^X_{x_1,x_2}&= \sum_{x_2}\sum_{y_1} p^U_{(x_1,y_1),\,(x_2,y_2)}\\
&=\frac{1}{|Y_2|} \sum_{x_2} \sum_{y_2'} \sum_{y_1} p^U_{(x_1,y_1),\,(x_2,y_2')}\\
& = \frac{1}{|Y_2|} \cdot |Y_1|=1
\end{align*}
where the third equality follows from independence of the sum over $Y_1$ from variables $y_2$. Similar arguments show that the relations hold for the rest of the matrices.

 Fix $x_1$ and let $x_2 \neq x_2'$ and any $ y_1$. Then,  
\begin{align*}
p^{X}_{x_1,x_2}p^{X}_{x_1,x_2'} & =  (\sum_{y_2}p^{Y}_{y_1,y_2}) p^{X}_{x_1,x_2}p^{X}_{x_1,x_2'}( \sum_{y_2'}p^{Y}_{y_1,y_2'})\\
& =\sum_{y_2, y_2'}p^{X}_{x_1,x_2}p^{Y}_{y_1,y_2} p^{X}_{x_1,x_2'}p^{Y}_{y_1,y_2'} \\
& = \sum_{y_2, y_2'} p^{U}_{(x_1,y_1),(x_2,y_2)}p^{U}_{(x_1,y_1),(x_2',y_2')}=0
\end{align*}
since $(x_2,y_2) \neq (x_2',y_2')$ where we used commutativity for the entries of $P^{X}, P^{Y}$ and column  orthogonality of the entries of $P^{U}$. Arguments for $P^Y, P^A, P^B$ follow by symmetry.

\end{proof}

\subsection{Simulation paradigm}
Let $U_i$ and $W_i$ be finite sets, $i=1,2$. A no--signalling (NS) correlation on the
quadruple $(U_2,W_1,U_1,W_2)$ is positive trace preserving map (information channel)
\[
\Gamma:\ \cl D_{U_2\times W_1}\longrightarrow \cl D_{U_1\times W_2}
\]
for which the marginal channels are well-defined, i.e., 
\[
\sum_{w_2\in W_2} \langle \delta_{u_1} \otimes \delta_{w_2}, \Gamma(\delta_{u_2} \otimes \delta_{w_1'})  \rangle , \qquad 
\sum_{u_1\in U_1} \langle \delta_{u_1} \otimes \delta_{w_2}, \Gamma(\delta_{u_2'} \otimes \delta_{w_1})  \rangle 
\]
 are independent of the choices of
$w_1'\in W_1$ and $u_2'\in U_2$, respectively. We note that this is equivalent with the definition of a no--signalling correlation as a family of conditional probability distributions that we used throughout the paper via the correspondence  
\[
\Gamma(u_1,w_2\mid u_2,w_1 ):= \langle \delta_{u_1} \otimes \delta_{w_2}, \Gamma(\delta_{u_2} \otimes \delta_{w_1})  \rangle 
\]

We recall a few notions from \cite{ht_one}. 
Given an NS correlation $\Gamma$ on the quadruple $(U_2,W_1,U_1,W_2)$ and a channel 
$E:\cl D_{U_1}\to \cl D_{W_1}$, define the linear map
\[
\Gamma[E]:\cl D_{U_2}\longrightarrow \cl D_{W_2}
\]
by
\begin{equation}\label{eq:GammaE}
\Gamma[E](w_2\mid u_2)
=\sum_{u_1\in U_1} \sum_{w_1\in W_1}
\Gamma(u_1,w_2\mid u_2,w_1)\,E(w_1\mid u_1).
\tag{2}
\end{equation} 
We say that a channel $F:\cl D_{U_2}\to \cl D_{W_2}$ is \emph{simulated} by $E$ with the assistance 
of $\Gamma$ if $F=\Gamma[E]$; in this case we call $\Gamma$ a \emph{simulator}.
The simulation procedure is illustrated by
\begin{center}
\begin{tikzcd}[column sep=large, row sep=large]
U_1 \arrow[r, "E"] & W_1 \arrow[d, dashed] \\
U_2 \arrow[u, dashed] \arrow[r, "{\Gamma[E]}"] & W_2
\end{tikzcd}
\end{center}

\medskip

Now we revisit the notion of strongly no-signalling correlations introduced in \cite{ht_one} (see also \cite{Gage}) and we give a slightly different definition that fits our context. Set \(X=A:= (X_1 \times Y_1) \sqcup (X_2 \times Y_2) \sqcup  (A_1\times B_1) \sqcup (A_2 \times B_2)\). 

We define a \emph{strongly no--signalling (SNS) correlation}  $ \Gamma $ over the quadruple $(X,A,X,A)$ to be a no-signalling correlation 
\[
\Gamma=\bigl(\Gamma(u_1,w_2|u_2,w_1)\bigr)_{u_1,u_2\in X, w_1,w_2\in A }
\]
satisfying the following conditions:
\begin{align*}
\sum_{x_1\in X_1} \Gamma((x_1,y_1),w_2\mid (x_2,y_2),w_1) 
&= \sum_{x_1\in X_1} \Gamma((x_1,y_1),w_2\mid (x_2',y_2),w_1), 
\; x_2,x_2'\in X_2, w_1,w_2 \in A,\\
\sum_{y_1\in Y_1} \Gamma((x_1,y_1),w_2\mid(x_2,y_2),w_1) 
&= \sum_{y_1\in Y_1} \Gamma((x_1,y_1),w_2\mid (x_2,y_2'),w_1), 
\;\; y_2,y_2'\in Y_2, w_1,w_2 \in A,\\[4pt]
\sum_{a_2\in A_2} \Gamma(u_1,(a_2,b_2)\mid u_2,(a_1,b_1)) 
&= \sum_{a_2\in A_2} \Gamma(u_1,(a_2,b_2)\mid u_2,(a_1',b_1)), 
\;\; a_1,a_1'\in A_1, u_1,u_2 \in X,\\[4pt]
\sum_{b_2\in B_2} \Gamma(u_1,(a_2,b_2)\mid u_2,(a_1,b_1)) 
&= \sum_{b_2\in B_2} \Gamma(u_1,(a_2,b_2)\mid u_2,(a_1,b_1')), 
\;\; b_1,b_1'\in B_1, u_1,u_2 \in X.
\end{align*}
We denote the different classes of SNS correlations by $ \cl C_{\rm st}$, where $ \rm t= \{loc, q, qa, qc, ns \}$.

\medskip

We fix finite sets $X$ and $Y$, and set
$A = X$ and $B = Y$. Let $p=\bigl(p(a,b\mid x,y)\bigr)_{a,b\in A,\;x,y\in X}$ be a no--signalling (NS) correlation 
over the quadruple $(X,Y,X,Y)$. 
We say that $p$ is an \emph{NS bicorrelation} \cite{bhtt} if the transpose family
\[
p^{\ast} \;=\; \bigl(p(a,b\mid x,y)\bigr)_{x,y}, \quad {\;a,b\in X\times  Y},
\]
is a conditional distribution with inputs $(a,b)\in A\times B$ and 
outputs $(x,y)\in X\times Y$, is also an NS correlation.

\medskip

Let $\Gamma=\bigl(\Gamma(u_1,w_2|u_2,w_1)\bigr)_{u_1,u_2\in X, w_1,w_2\in A }$  be an SNS correlation over $(X,A,X,A)$.
We say that $\Gamma$ is an \emph{SNS bicorrelation}, if it is a bicorrelation and $ \Gamma^*$ is also SNS, that is, the symmetric strong marginal equalities hold.

\begin{remark} \rm  
In \cite{Gage, ht_one}, the authors defined  a strongly no-signalling (SNS) correlation as a no signalling correlation $ \Gamma$ over $(X_2\times Y_2,A_1\times B_1,X_1\times Y_1,A_2\times B_2)$ that satisfies the strong NS conditions. 
Our definition agrees for perfect strategies of the hypergraph isomorphism game by restricting to the specific blocks. Indeed, in our setting the question and answer alphabets are the large disjoint unions 
\( 
X=A=(X_1\times Y_1)\sqcup(X_2\times Y_2)\sqcup(A_1\times B_1)\sqcup(A_2\times B_2).
\) 
A perfect strategy for the hypergraph isomorphism game respects the swap/type rule, so whenever the inputs lie in 
$(X_1\times Y_1)$ and $(X_2\times Y_2)$ the outputs are forced into $(A_2\times B_2)$ and $(A_1\times B_1)$, with all 
other blocks occurring with probability zero. If we restrict our correlation to this support and simply relabel the 
coordinates, we obtain a correlation on 
\( 
(X_2\times Y_2,\;A_1\times B_1,\;X_1\times Y_1,\;A_2\times B_2),
\)
which is exactly the domain considered in \cite{Gage, ht_one}. Moreover, the strong no--signalling equalities we imposed 
reduce to theirs on this block. In a similar fashion, a perfect SNS bicorrelation reduces to a SNS bicorrelation in the sense of \cite{Gage, ht_one}. 
\end{remark}

We note that SNS correlations act as transporters in the following sense.
\begin{theorem}\cite[Theorem 6.4]{ht_one} \label{th:transportation}
Let $\Gamma$ be an SNS correlation over the quadruple $(X_2\times Y_2,A_1 \times B_1,X_1 \times Y_1,A_2 \times B_2)$
and let $\cl E$ be an NS correlation over $(X_1,Y_1,A_1,B_1)$. Then $\Gamma[\cl E]$ is NS correlation over $(X_2,Y_2,A_2,B_2)$ and,
\begin{enumerate}
  \item[(i)] if $\Gamma \in \mathcal{C}^{\mathrm{}}_{\mathrm{sqc}}$ and $\cl E \in \mathcal{C}^{\mathrm{}}_{\mathrm{qc}}$, then $\Gamma[\cl E] \in \mathcal{C}^{\mathrm{}}_{\mathrm{qc}}$;
  \item[(ii)] if $\Gamma \in \mathcal{C}^{\mathrm{}}_{\mathrm{sqa}}$ and $\cl E \in \mathcal{C}^{\mathrm{}}_{\mathrm{qa}}$, then $\Gamma[\cl E] \in \mathcal{C}^{\mathrm{}}_{\mathrm{qa}}$;
  \item[(iii)] if $\Gamma \in \mathcal{C}^{\mathrm{}}_{\mathrm{sq}}$ and $\cl E \in \mathcal{C}^{\mathrm{}}_{\mathrm{q}}$, then $\Gamma[\cl E] \in \mathcal{C}^{\mathrm{}}_{\mathrm{q}}$;
  \item[(iv)] if $\Gamma \in \mathcal{C}^{\mathrm{}}_{\mathrm{sloc}}$ and $\cl E \in \mathcal{C}^{\mathrm{}}_{\mathrm{loc}}$, then $\Gamma[\cl E] \in \mathcal{C}^{\mathrm{}}_{\mathrm{loc}}$.
\end{enumerate}
\end{theorem}

\medskip

We make the following modification of \cite[Definition 6.5]{ht_one}:
\begin{definition} \label{def:gameisom} \rm 
We call two non-local games \(\cl G_{i}=(X_i,Y_i;A_i,B_i,\lambda_i)\)  \emph{$\rm t$--isomorphic}, and  write $ \cl G_{1}\cong_{\rm t}\cl  G_{2}$ if \(|X_1|=|X_2|\), \(|Y_1|=|Y_2|\), \(|A_1|=|A_2|\), \(|B_1|=|B_2|\) and the hypergraph isomorphism game $ (\Lambda_{\cl G_1}, \Lambda_{\cl G_2})$ has a perfect SNS strategy $ \Gamma$ of type $ \rm t$, where $ \rm t= \{loc, q, qa, qc\}$. 

\end{definition}

\begin{theorem} \label{th:NSgameisom}
Let \(\cl G_{i}=(X_i,Y_i;A_i,B_i,\lambda_i)\) be two non-local games with  \(|X_1|=|X_2|\), \(|Y_1|=|Y_2|\), \(|A_1|=|A_2|\), \(|B_1|=|B_2|\).
 We have the following:
\begin{itemize}
    \item[(i)] $ \cl G_{1}\cong_{\rm loc} \cl G_{2}$ if and only if there exists a unital $*$-homomorphism $\mathcal O_{\mathrm{NS}}(\cl G_1,\cl G_2)\rightarrow \bb{C}$.
    \item[(ii)] $ \cl G_{1}\cong_{\rm q}\cl G_{2}$ if and only if there exists a unital $*$-homomorphism $\mathcal O_{\mathrm{NS}}(\cl G_1,\cl G_2)\rightarrow \cl{B}(H)$ for some non-zero, finite-dimensional Hilbert space $H$.
    \item[(iii)] $ \cl G_{1}\cong_{\rm qa} \cl G_{2}$  if and only if there exists a unital $*$-homomorphism $\mathcal O_{\mathrm{NS}}(\cl G_1,\cl G_2) \rightarrow \cl{R}^{\cl{U}}$, where $\cl{U}$ is a free ultrafilter along $\bb{N}$ and $\cl{R}$ is the hyperfinite ${\rm II}_{1}$ factor.
    \item[(iv)] $ \cl G_{1}\cong_{\rm qc}\cl  G_{2}$ if and only if there exists a unital $*$-homomorphism $\mathcal O_{\mathrm{NS}}(\cl G_1,\cl G_2) \rightarrow A$, where $A$ is a unital ${\rm C}^{*}$-algebra endowed with a faithful tracial state.
\end{itemize}
\end{theorem}
\begin{proof}
    We will show the case $ \rm t =qc$; the others will follow by similar arguments. Recall that $\mathcal O_{\mathrm{NS}}(\cl G_1,\cl G_2) = \mathcal O(\Lambda_{\cl G_1}, \Lambda_{\cl G_2})\big/ I_{\mathrm{NS}}$ where $ O(\Lambda_{\cl G_1}, \Lambda_{\cl G_2})$ is the game $*$-algebra for the hypergraph isomorphism game between the non-local games $ \cl G_1$ and $ \cl G_2$. 

    $ (\implies )$ Let $ \Gamma$ be a perfect SNS $\rm qc$-strategy for the hypergraph isomorphism game. Using \cite[Theorem 5.5.]{quantum-chromatic} (see also \cite[Theorem 3.2]{hmps}), there exists a C*-algebra $A$ equipped with a faithful tracial state $ \tau$ such that 
    \[
    \Gamma( x,y |a,b )= \tau (u_{x,a} u_{y,b})  
    \]
    where $u_{x,a}, u_{y,b}$ are PVM's  satisfying the rules of the game. Recall that  $a,b,x,y$ vary over \( (X_1 \times Y_1) \sqcup (X_2 \times Y_2) \sqcup  (A_1\times B_1) \sqcup (A_2 \times B_2)\).  The arguments in Proposition \ref{prop:game-alg-presentation} and the fact that $\tau$ is faithful show that the projections $ u_{x,a}$ give rise to magic unitaries $\bigl(p^U_{(x_1,y_1),\,(x_2,y_2)}\bigr) $ and  $\bigl(p^W_{(a_1,b_1),\,(a_2,b_2)}\bigr)_{} $ that intertwine with the incidence matrices.  For $x_{2}, x_{2}' \in X_{2}$ set
 \[
S \;:=\;
\sum_{x_1\in X_1} p^U_{(x_1,y_1),(x_2,y_2)}
\;-\;
\sum_{x_1\in X_1} p^U_{(x_1,y_1),(x_2',y_2)}.
\]
Using the fact that $\Gamma$ is strongly no-signalling, we will show that $ \tau (S^2)=0$ implying that $ S=0$ since the trace is faithful. The rest of the NS conditions will follow from similar arguments. We calculate: 
\begin{eqnarray*}
& &  
\tau\!\left(\left(\sum_{x_1\in X_1} p^U_{(x_1,y_1),(x_2,y_2)}
\;-\;
\sum_{x_1\in X_1} p^U_{(x_1,y_1),(x_2',y_2)}\right)^2\right)\\
&= & 
\!\!\!\!\! \!\!\!\sum_{x_1,x_1'\in X_1} \!\!\! \tau\!\big(p^U_{(x_1,y_1),(x_2,y_2)}\,p^U_{(x_1',y_1),(x_2,y_2)}\big) 
- 
\!\!\! \!\!\!\sum_{x_1,x_1'\in X_1} \!\!\!\tau\!\big(p^U_{(x_1,y_1),(x_2',y_2)}\,p^U_{(x_1',y_1),(x_2,y_2)}\big)\\
& - &  
\!\!\!\!\! \!\!\!\sum_{x_1,x_1'\in X_1}\!\!\! \tau\!\big(p^U_{(x_1,y_1),(x_2,y_2)}\,p^U_{(x_1',y_1),(x_2',y_2)}\big)
+ 
\!\!\!\!\!\!\sum_{x_1,x_1'\in X_1} \!\!\!\tau\!\big(p^U_{(x_1,y_1),(x_2',y_2)}\,p^U_{(x_1',y_1),(x_2',y_2)}\big)\\[2mm]
& = & 
\!\!\!\!\!\!\!\!\sum_{x_1,x_1'\in X_1} \!\!\!\Gamma\! \big((x_1,y_1),(x_1',y_1)\big|(x_2,y_2),(x_2,y_2)\big)
- \!\!\!\!\!\!
\sum_{x_1,x_1'\in X_1}\!\!\! \Gamma\!\big((x_1,y_1),(x_1',y_1)\big|(x_2',y_2),(x_2,y_2)\big)\\
& - &  
\!\!\!\!\!\!\!\!\sum_{x_1,x_1'\in X_1} \!\!\!\Gamma\!\big((x_1,y_1),(x_1',y_1)\big|(x_2,y_2),(x_2',y_2)\big)
+ 
\!\!\!\!\!\sum_{x_1,x_1'\in X_1}\!\!\! \Gamma\!\big((x_1,y_1),(x_1',y_1)\big|(x_2',y_2),(x_2',y_2)\big)\\[2mm]
&=& \!\!\!\sum_{x_1'}\!\!\Bigg[
\sum_{x_1}\Gamma\!\big((x_1,y_1),(x_1',y_1)\big|(x_2,y_2),(x_2,y_2)\big)
-\!\!\sum_{x_1}\Gamma\!\big((x_1,y_1),(x_1',y_1)\big|(x_2',y_2),(x_2,y_2)\big)
\Bigg] \\
& -& \!\!\!\sum_{x_1'}\!\!\Bigg[
\sum_{x_1}\Gamma\!\big((x_1,y_1),(x_1',y_1)\big|(x_2,y_2),(x_2',y_2)\big)
-\!\!\sum_{x_1}\Gamma\!\big((x_1,y_1),(x_1',y_1)\big|(x_2',y_2),(x_2',y_2)\big)
\Bigg]\\[2mm]
& = & 0.
\end{eqnarray*}

As this holds for any $x_{2}, x_{2}' \in X_{2}$, we conclude that there exists a unital $*$-homomorphism $\mathcal O_{\mathrm{NS}}(\cl G_1,\cl G_2) \rightarrow A$.

    $( \impliedby) $ By assumption, and Lemma \ref{l:NS-decomp},  we have  $ (p^X_{x_1,x_2})_{x_1 \in X_1}$ , $ (p^Y_{y_1,y_2})_{y_1\in Y_1}$, $ (p^A_{a_1,a_2})_{a_2\in A_2}$ and $ (p^B_{b_1,b_2})_{b_2 \in B_2}$ are magic unitaries for all $ x_2, y_2, a_1, b_1$ resp. with entries in a unital tracial $C^*$-algebra $A$ and 
\[
p^U_{(x_1,y_1),\,(x_2,y_2)} = p^X_{x_1,x_2}\,p^Y_{y_1,y_2},
\qquad
p^W_{(a_1,b_1),\,(a_2,b_2)} = p^A_{a_1,a_2}\,p^B_{b_1,b_2}.
\]
Now setting $ \Gamma( x,y |a,b )= \tau (u_{x,a} u_{y,b}) $ where $u_{x,a}$ is defined via the generators of $ \cl O_{\rm NS}(\cl G_1,\cl G_2)$ as in Proposition \ref{prop:game-alg-presentation} defines a perfect $\rm qc$-correlations that is moreover SNS by the above decomposition of the generators in the $*$-algebra.

\end{proof}

Let \(\cl G_{i}=(X_i,Y_i;A_i,B_i,\lambda_i)\) be two non-local games with  \(|X_1|=|X_2|\), \(|Y_1|=|Y_2|\), \(|A_1|=|A_2|\), \(|B_1|=|B_2|\). Motivated by the above Theorem we may also say that $ \cl G_1, \cl G_2$ are \textit{$A^*$-isomorphic} and write $ \cl G_1 \cong_{A^*} \cl G_2$ if $\mathcal O_{\mathrm{NS}}(\cl G_1,\cl G_2) \neq 0 $
and that $ \cl G_1, \cl G_2$ are \textit{$C^*$-isomorphic} and write $ \cl G_1 \cong_{C^*} \cl G_2$ if $\mathcal O_{\mathrm{NS}}(\cl G_1,\cl G_2)  $ has a non zero representation into  $ \cl B(H)$ for some Hilbert space $H$.

\begin{corollary}
     Let \(\cl G_{i}=(X_i,Y_i;A_i,B_i,\lambda_i)\)  be two non-local games  such that $ \cl{G}_{1}\cong_{\rm t} \cl{G}_{2}$, $ \rm t= loc,q,qa,qc$.  Then $ \cl G_1$ has a perfect $\rm t$-strategy if and only if $\cl G_2$ has a perfect $\rm t$-strategy.
\end{corollary}
\begin{proof}
     Since $ \cl{G}_{1}\cong_{\rm t} \cl{G}_{2}$, means by definition that there exists a perfect SNS strategy of type $t$ coming from the respective representation of the $*$-algebra $\cl O_{NS}(\Lambda_{\cl G_1}, \Lambda_{\cl G_2})$, the ``only if" follows from Theorem \ref{th:transportation} and \cite[Proposition 3.2]{ht_one}.
     
     On the other hand,  Lemma \ref{l:NS-decomp} implies that the induced SNS correlation is in fact an SNS bicorrelation $\Gamma$, and the ``if"  follows from Theorem \ref{th:transportation} and \cite[Proposition 3.2]{ht_one} applied to $\Gamma$ and $ \Gamma^*$.
\end{proof}

\begin{remark} \rm In \cite{Gage, ht_one}, the authors defined  different classes of SNS bicorrelations via no-signalling operator matrices; that is, matrices of operators satisfying no-signalling conditions similar to  those for SNS correlations. We note that
when a perfect SNS qc-bicorrelation $\Gamma$ has a realization via magic unitaries $ P_{(x_1,y_1),(x_2,y_2)}, Q_{(a_1,b_2),(a_2,b_2)}$ that satisfy   no-signalling conditions, we will automatically have \(|X_1|=|X_2|\), \(|Y_1|=|Y_2|\), \(|A_1|=|A_2|\), \(|B_1|=|B_2|\). On the other hand, requiring all of the above sets to be equal makes a perfect SNS qc-correlation $\Gamma$ an SNS bicorrelation by Lemma \ref{l:NS-decomp} as we saw  in Theorem \ref{th:NSgameisom}.  For this reason, we did not require bicorrelations in the Definition \ref{def:gameisom} in contrast to  \cite[Definition 6.5]{ht_one}.
\end{remark}

We now demonstrate a simple example.
\begin{example} \rm 
Let $\cl G_i=(X_i,Y_i;A_i,B_i,\lambda_i)$ be two nonlocal games and assume that there exist bijections
\[
\alpha_X:X_1\!\to\!X_2,\ \alpha_Y:Y_1\!\to\!Y_2,\ 
\beta_A:A_1\!\to\!A_2,\ \beta_B:B_1\!\to\!B_2
\]
such that for all $x,y,a,b$,
\[
\lambda_2\big(\alpha_X(x),\alpha_Y(y)\,\big|\,\beta_A(a),\beta_B(b)\big)=\lambda_1(x,y\,|\,a,b).
\]
Write $U_i:=X_i\times Y_i$, $W_i:=A_i\times B_i$ and let $A_{\Lambda_i}$ be the incidence matrices of the
hypergraphs $\Lambda_{\cl G_i}\subseteq U_i\times W_i$.

Define permutation blocks $P_U\in M_{U_1\times U_2}$, $P_W\in M_{W_1\times W_2}$ by
\[
p^U_{(x_1,y_1),(x_2,y_2)}=\delta_{x_2,\alpha_X(x_1)}\delta_{y_2,\alpha_Y(y_1)},\qquad
p^W_{(a_1,b_1),(a_2,b_2)}=\delta_{a_2,\beta_A(a_1)}\delta_{b_2,\beta_B(b_1)}.
\] 
Then $P:=P_U\oplus P_W$ is unitary satisfying the strong NS conditions and 
\[
A_{\Lambda_1}\,P_W \;=\; P_U\,A_{\Lambda_2},
\]
by verifying Lemma \ref{l_inclusion_rel}.
Hence the entries of $P_U,P_W$ give a  classical perfect SNS strategy
for the hypergraph–isomorphism game $(\Lambda_{\cl G_1},\Lambda_{\cl G_2})$. Consequently $\cl G_1\cong_{\mathrm{loc}}\cl G_2$.
\end{example}

\subsection{The NS algebra as a quotient bi-Galois extension}

Given that the NS game $*$-algebra $\mathcal O_{\mathrm{NS}}(\cl G_1,\cl G_2)$ is a quotient of a bi-Galois extension (namely of ${\cl O}(\Lambda_{\cl G_1},\Lambda_{\cl G_2})$) we wish to show next that the former also carries a natural bi-Galois extension structure.

Denote the fundamental representations on the question/answer blocks by
\[
u^U=\big(u^U_{(x_1,y_1),(x_1',y_1')}\big)_{U_1\times U_1},\quad
u^W=\big(u^W_{(a_1,b_1),(a_1',b_1')}\big)_{W_1\times W_1}
\ \text{in }\cl O(\Lambda_{\cl G_1}),
\]
\[
v^U=\big(v^U_{(x_2,y_2),(x_2',y_2')}\big)_{U_2\times U_2},\quad
v^W=\big(v^W_{(a_2,b_2),(a_2',b_2')}\big)_{W_2\times W_2}
\ \text{in }\cl O(\Lambda_{\cl G_2}).
\]

\medskip
Set $J_1 $  to be the $*$-ideal in $ \cl O(\Lambda_{\cl G_1})$  generated by the four
families of NS relations on each block. For $U$:
\[
\begin{aligned}
&\sum_{x_1} u^U_{(x_1,y_1),(x_1',y_1')}=\sum_{x_1} u^U_{(x_1,y_1),(x_1'',y_1')}
,\qquad
\sum_{x_1'} u^U_{(x_1,y_1),(x_1',y_1')}=\sum_{x_1'} u^U_{(x_1,y_1),(x_1',y_1'')},
\\[-2pt]
&\sum_{y_1} u^U_{(x_1,y_1),(x_1',y_1')}=\sum_{y_1} u^U_{(x_1,y_1),(x_1',y_1'')}
,\qquad
\sum_{y_1'} u^U_{(x_1,y_1),(x_1',y_1')}=\sum_{y_1'} u^U_{(x_1,y_1),(x_1'',y_1')},
\end{aligned}
\]
and for $W$:
\[
\begin{aligned}
&\sum_{a_1'} u^W_{(a_1,b_1),(a_1',b_1')}=\sum_{a_1'} u^W_{(a_1'',b_1),(a_1',b_1')}
,\qquad
\sum_{a_1} u^W_{(a_1,b_1),(a_1',b_1')}=\sum_{a_1} u^W_{(a_1,b_1''),(a_1',b_1')},
\\[-2pt]
&\sum_{b_1} u^W_{(a_1,b_1),(a_1',b_1')}=\sum_{b_1} u^W_{(a_1,b_1),(a_1',b_1'')}
,\qquad
\sum_{b_1'} u^W_{(a_1,b_1),(a_1',b_1')}=\sum_{b_1'} u^W_{(a_1,b_1),(a_1'',b_1')}.
\end{aligned}
\]

Analogously we define  $J_2 $  in $\cl O(\Lambda_{\cl G_2})$.
Set
\[
H_1:=\cl O(\Lambda_{\cl G_1})/J_1,\qquad
H_2:=\cl O(\Lambda_{\cl G_2})/J_2.
\]

For the following facts we refer the reader to \cite[Theorem 4.2.1]{abe} and the preceeding discussion. Let $(H,\Delta,\varepsilon,S)$ be a Hopf $*$-algebra over $\mathbb C$.
A \emph{Hopf $*$-ideal} $J\subseteq H$ is a two–sided $*$-ideal such that
\[
\Delta(J)\ \subseteq\ J \otimes A+ A\otimes J,\qquad
\varepsilon(J)=0,\qquad
S(J)\subseteq J.
\]

Let $H$ be a Hopf $*$-algebra and $ J$ be a Hopf $*$-ideal.
Then the quotient algebra \(\overline H:=H/J\) is a Hopf $*$-algebra with structure maps
\[
\overline\Delta\bigl( q(h)\bigr):=(q\otimes q)\Delta(h),\qquad
\overline\varepsilon\bigl( q(h)\bigr):=\varepsilon(h),\qquad
\overline S\bigl( q(h)\bigr):=q\bigl(S(h)\bigr),
\]
where \(q:H\twoheadrightarrow \overline H\) is the quotient map.

\begin{lemma}\label{lem:JL-Hopf}
For $i=1,2$, let $\mathcal O(\Lambda_{\mathcal G_i})$ be the Hopf $*$-algebra of the quantum
automorphism group of $\Lambda_{\mathcal G_i}$ and let $J_i$ be the $*$-ideal defined above.
Then $H_i:=\mathcal O(\Lambda_{\mathcal G_i})/J_i$ is a Hopf $*$-algebra. In particular, if $\pi_i:\mathcal O(\Lambda_{\mathcal G_i})\twoheadrightarrow H_i$ denotes the quotient map,
there is a \emph{unique} Hopf $*$-algebra structure on $H_i$ such that
\[
\overline\Delta_i\bigl(\pi_i(u)\bigr)=(\pi_i\otimes \pi_i)\Delta(u),\qquad
\overline\varepsilon_i\bigl(\pi_i(u)\bigr)=\varepsilon(u),\qquad
\overline S_i\bigl(\pi_i(u)\bigr)=\pi_i\bigl(S(u)\bigr),
\]
for all $u\in\mathcal O(\Lambda_{\mathcal G_i})$.
\end{lemma}

\begin{proof}
We only treat  the comultiplication $\Delta$ case. The rest of them are straightforward.
Work on the $U$–block; the $W$–block is identical. Fix $x,x'\in X_1$ and $y',y''\in Y_1$ and consider the NS fiber
generator
\[
R^U_{x,x';\,y',y''}
\;:=\;
\sum_{y\in Y_1} u^{U}_{(x,y),(x',y')}
\;-\;
\sum_{y\in Y_1} u^{U}_{(x,y),(x',y'')}\ \in\ J_1.
\]

The comultiplication on matrix coefficients is
\(
\Delta\big(u^{U}_{\alpha,\beta}\big)=\sum_{\gamma\in U_1} u^{U}_{\alpha,\gamma}\otimes u^{U}_{\gamma,\beta}.
\)
Thus
\[
\begin{aligned}
\Delta\!\big(R^U_{x,x';\,y',y''}\big)
&=
\sum_{y\in Y_1}\sum_{(r,s)\in U_1}
u^{U}_{(x,y),(r,s)}\otimes u^{U}_{(r,s),(x',y')}
\;-\;
\sum_{y\in Y_1}\sum_{(r,s)\in U_1}
u^{U}_{(x,y),(r,s)}\otimes u^{U}_{(r,s),(x',y'')}
\\[2mm]
&=
\sum_{(r,s)\in U_1}
\Big(\sum_{y\in Y_1} u^{U}_{(x,y),(r,s)}\Big)
\ \otimes\
\Big(u^{U}_{(r,s),(x',y')}-u^{U}_{(r,s),(x',y'')}\Big).
\end{aligned}
\]
For each $(r,s)$,
\[
\pi_1\Big(\sum_{y\in Y_1} u^U_{(x,y),(r,s)}\Big)\ =\ \pi_1\big(u^{X_1}_{x,r}\big),
\]
which is independent of $s$. Apply $\pi_1\otimes \pi_1$  and group the sum over $s$:
\[
\begin{aligned}
(\pi_1\otimes \pi_1)\Delta(R_{x, x';y',y''}^{U})
&=\sum_{(r,s)} \pi_1\Big(\sum_{y} u^U_{(x,y),(r,s)}\Big)\ \otimes\ \pi_1\Big(u^U_{(r,s),(x',y')}-u^U_{(r,s),(x',y'')}\Big)\\
&=\sum_{r\in X_1} \pi_1 \big(u^{X_1}_{x,r}\big)\ \otimes\
\pi_1\Big(\sum_{s\in Y_1} u^U_{(r,s),(x',y')}-\sum_{s\in Y_1} u^U_{(r,s),(x',y'')}\Big).
\end{aligned}
\]
The bracket on the right is exactly the $Y$–fiber NS generator with left index $r$, hence lies in $J_1$ and vanishes under $\pi_1$.
Therefore $(\pi_1\otimes \pi_1)\Delta(R_{x, x';y',y''}^{U})=0$, for every $x, x' \in X_{1}, y', y'' \in Y_{1}$.
This proves the coideal part of the Hopf ideal property for (one family of) generators.
\end{proof}

\begin{theorem}\label{thm:NS-bigalois-quotients-kappa-kernels-commute}
Let $\cl G_i=(X_i,Y_i;A_i,B_i,\lambda_i)$, $i=1,2$, be two non-local games.
Let $J_i \subseteq \cl O(\Lambda_{\cl G_i})$ be the $*$-ideals generated by the four NS  families,
and put $H_i:=\cl O(\Lambda_{\cl G_i})/J_i$ with quotient maps $\pi_i:\cl O(\Lambda_{\cl G_i})\twoheadrightarrow H_i$.
Write
\[
q:\ \cl O(\Lambda_{\cl G_1},\Lambda_{\cl G_2})\ \longrightarrow\
\cl O_{\mathrm{NS}}:=\cl O(\Lambda_{\cl G_1},\Lambda_{\cl G_2})/I_{\mathrm{NS}}
\]
for the NS quotient. Then $\cl O_{\mathrm{NS}}$ is an $H_1$–$H_2$ bi-Galois extension.  
The descended coactions are given by
\[
\overline\delta\circ q \;=\; (\pi_1\otimes q)\circ \delta:
\cl O(\Lambda_{\cl G_1},\Lambda_{\cl G_2})\longrightarrow H_1\otimes \cl O_{\mathrm{NS}},
\]
\[
\overline\gamma\circ q \;=\; (q\otimes \pi_2)\circ \gamma:
\cl O(\Lambda_{\cl G_1},\Lambda_{\cl G_2})\longrightarrow \cl O_{\mathrm{NS}}\otimes H_2,
\]
where $\delta,\gamma$ are the unreduced commuting coactions.
\end{theorem}

\begin{proof}

By Lemma~\ref{lem:JL-Hopf}, $J_i$ are Hopf $*$-ideals in $\cl O(\Lambda_{\cl G_i})$, hence $H_i$ inherit Hopf $*$-structures. Let
\[
\delta:\ \cl O(\Lambda_{\cl G_1},\Lambda_{\cl G_2})\longrightarrow \cl O(\Lambda_{\cl G_1})\otimes \cl O(\Lambda_{\cl G_1},\Lambda_{\cl G_2}),\qquad
\gamma:\ \cl O(\Lambda_{\cl G_1},\Lambda_{\cl G_2})\longrightarrow \cl O(\Lambda_{\cl G_1},\Lambda_{\cl G_2})\otimes \cl O(\Lambda_{\cl G_2})
\]
be the unreduced commuting coactions.

As in the proof of Lemma~\ref{lem:JL-Hopf}, checking on NS fiber generators (on both $U$ and $W$ blocks) yields
\begin{equation}\label{eq:kern-left-right}
\delta(I_{\mathrm{NS}})\ \subseteq\ \ker(\pi_1\otimes q),\qquad
\gamma(I_{\mathrm{NS}})\ \subseteq\ \ker(q\otimes \pi_2).
\end{equation}
On the $U$-block (the $W$-block is identical), take the NS fiber generator
\[
R^U_{x,x';\,y',y''}
\ :=\
\sum_{y\in Y_1} p^U_{(x,y),(x',y')}
\;-\;
\sum_{y\in Y_1} p^U_{(x,y),(x',y'')}
\ \in I_{\mathrm{NS}}.
\]
Since
\[
\delta\!\big(p^U_{\alpha,\beta}\big)
=\sum_{\gamma\in U_1}u^U_{\alpha,\gamma}\otimes p^U_{\gamma,\beta},
\]
we obtain
\[
\delta\!\big(R^U_{x,x';\,y',y''}\big)
=\sum_{(r,s)\in U_1}
\Big(\sum_{y\in Y_1}u^U_{(x,y),(r,s)}\Big)\ \otimes\
\Big(p^U_{(r,s),(x',y')}\;-\;p^U_{(r,s),(x',y'')}\Big).
\]

Apply $\pi_1\otimes q$. By construction of $J_1$, the left factor
$\sum_{y}u^U_{(x,y),(r,s)}$ is independent of $s$ modulo $J_1$, while the right factor is an NS fiber difference and hence vanishes under $q$.
Thus $(\pi_1\otimes q)\circ\delta\!\big(R^U_{x,x';\,y',y''}\big)=0$. The same argument for the right coaction $\gamma$ yields $ \gamma(I_{\mathrm{NS}})\ \subseteq\ \ker(q\otimes \pi_2)$.

\smallskip

By \eqref{eq:kern-left-right}, the formulas
\[
\overline\delta\circ q=(\pi_1\otimes q)\circ \delta,\qquad
\overline\gamma\circ q=(q\otimes \pi_2)\circ \gamma
\]
give rise to well–defined $*$-homomorphisms
$\overline\delta:\cl O_{\mathrm{NS}}\to H_1\otimes \cl O_{\mathrm{NS}}$ and
$\overline\gamma:\cl O_{\mathrm{NS}}\to \cl O_{\mathrm{NS}}\otimes H_2$.
Coassociativity and the counit descend directly from those of $\delta,\gamma$.

For commutation, recall (Theorem \ref{th_bigalois_ext}) the unreduced identity
$(\delta\otimes \id)\circ\gamma=(\id\otimes \gamma)\circ\delta$ on $\cl O(\Lambda_{\cl G_1},\Lambda_{\cl G_2})$.
Apply $\pi_1\otimes q\otimes \pi_2$ and use \eqref{eq:kern-left-right}:
for every $a$,
\[
(\pi_1\otimes q\otimes \pi_2)\circ(\delta\otimes \id)\circ\gamma(a)
=(\pi_1\otimes q\otimes \pi_2)\circ(\id\otimes \gamma)\circ\delta(a).
\]
By the definitions of the descended coactions this is exactly
\[
(\overline\delta\otimes \id)\circ\overline\gamma\big(q(a)\big)
=(\id\otimes \overline\gamma)\circ\overline\delta\big(q(a)\big),
\]
so $(\overline\delta\otimes \id)\circ\overline\gamma=(\id\otimes \overline\gamma)\circ\overline\delta$ on $\cl O_{\mathrm{NS}}$.

\smallskip

Let
\[
\kappa_\delta:\ \cl O(\Lambda_{\cl G_1},\Lambda_{\cl G_2})\otimes \cl O(\Lambda_{\cl G_1},\Lambda_{\cl G_2})\longrightarrow \cl O(\Lambda_{\cl G_1})\otimes \cl O(\Lambda_{\cl G_1},\Lambda_{\cl G_2}),\qquad
\kappa_\delta(a\otimes b)=\delta(a)(1\otimes b)
\]
be the unreduced left canonical map. Define the descended left canonical map
\[
\overline\kappa_\delta:\ \cl O_{\mathrm{NS}}\otimes \cl O_{\mathrm{NS}}
\longrightarrow H_1\otimes \cl O_{\mathrm{NS}},
\qquad
\overline\kappa_\delta( a\otimes  b)
:= \overline\delta( a)\,(1\otimes  b).
\]
Equivalently, $\overline\kappa_\delta$ is
the unique linear map making the square commute:
\[
(\pi_1\otimes q)\circ \kappa_\delta\ =\ \overline\kappa_\delta\circ (q\otimes q).
\]
Define a unital $*$-homomorphism
\[
\alpha:\ \cl O(\Lambda_{\cl G_1})\longrightarrow \cl O(\Lambda_{\cl G_1},\Lambda_{\cl G_2})\otimes \cl O(\Lambda_{\cl G_1},\Lambda_{\cl G_2}),\qquad
\alpha(u_{\alpha, \beta}^{U})=\sum_{\gamma \in U_{2}}p_{\alpha, \gamma}^{U}\otimes p_{\beta, \gamma}^{U},
\]
for $\alpha, \beta \in U_{1}$ (and defined similarly on the $u^{W}$ block generators). Furthermore, set
\[
\eta_\delta\ :=\ (\id\otimes m)\circ(\alpha\otimes \id):\
\cl O(\Lambda_{\cl G_1})\otimes \cl O(\Lambda_{\cl G_1},\Lambda_{\cl G_2})\longrightarrow \cl O(\Lambda_{\cl G_1},\Lambda_{\cl G_2})\otimes \cl O(\Lambda_{\cl G_1},\Lambda_{\cl G_2}),
\]
where $m$ is the multiplication map  and recall from the proof of Theorem \ref{th_bigalois_ext} that $\eta_\delta$ is the inverse of $\kappa_\delta$. Similar arguments as in the proof above shows that  we have
\[
(q\otimes q)\circ \alpha=\bar\alpha\circ \pi_1
\]
for a well-defined $*$-homomorphism $\bar\alpha:H_1\to \cl O_{\mathrm{NS}}\otimes \cl O_{\mathrm{NS}}$.
Define
\[
\overline\eta_\delta\ :=\ (\id\otimes \bar m)\circ(\bar\alpha\otimes \id):\
H_1\otimes \cl O_{\mathrm{NS}}\longrightarrow \cl O_{\mathrm{NS}}\otimes \cl O_{\mathrm{NS}},
\]
where $\bar m$ is multiplication on $\cl O_{\mathrm{NS}}$.
It is now straightforward to verify
\[
\overline\eta_\delta\circ \overline\kappa_\delta=\id,
\qquad
\overline\kappa_\delta\circ \overline\eta_\delta=\id,
\]
so $\overline\kappa_\delta:\cl O_{\mathrm{NS}}\otimes \cl O_{\mathrm{NS}}\xrightarrow{ } H_1\otimes \cl O_{\mathrm{NS}}$ is a bijection. (The argument for $\overline \kappa_\gamma$ and $H_2$ is identical.)

\end{proof}

\begin{theorem} 
    Let $\cl G_i=(X_i,Y_i;A_i,B_i,\lambda_i)$, $i=1,2$, be two non-local games. If $\mathcal O_{\mathrm{NS}}(\cl G_1,\cl G_2)\neq 0$ then it admits a non-zero $*$-representation as bounded operators on a Hilbert space.
\end{theorem}

\begin{proof}
The stated result follows by using Theorem \ref{thm:NS-bigalois-quotients-kappa-kernels-commute}, diagram chasing,
    \[
\begin{tikzcd}[column sep=huge, row sep=large]
A_u^0(n,m) \arrow[r, two heads, "\ \sigma\ "] \arrow[d, "\alpha_{n,m}"']
& \cl O(\Lambda_1,\Lambda_2) \arrow[r, two heads, "q"] \arrow[d, "\delta"]
& \cl O_{\mathrm{NS}} \arrow[d, "\overline\delta"] \\
A_u^0(n)\otimes A_u^0(n,m) \arrow[r, two heads, "\ \pi\otimes\sigma\ "']
& \cl O(\Lambda_1)\otimes \cl O(\Lambda_1,\Lambda_2) \arrow[r, two heads, "\ \pi_1\otimes q\ "']
& H_1\otimes \cl O_{\mathrm{NS}}
\end{tikzcd}
\]
and  invoking \cite[Theorem 6.2.6]{bichon99} as in the proof of Theorem \ref{t_nonzero_implies_trace}.

\end{proof}

\begin{corollary}
    Let $\cl G_i=(X_i,Y_i;A_i,B_i,\lambda_i)$, $i=1,2$, be two non-local games with \(|X_1|=|X_2|\), \(|Y_1|=|Y_2|\), \(|A_1|=|A_2|\), \(|B_1|=|B_2|\). Then 
    \[
    \cl G_1 \cong_{A^*} \cl G_2 \Longleftrightarrow \cl G_1 \cong_{{\rm C}^*} \cl G_2 \Longleftrightarrow \cl G_1 \cong_{\rm qc} \cl G_2.
    \]
\end{corollary}


\end{document}